\documentclass[10pt]{article}

\usepackage[utf8]{inputenc}
\usepackage{hyperref}
\usepackage{amsthm}
\usepackage{amsmath}
\usepackage{enumitem}
\usepackage{amssymb}
\usepackage{tikz}
\usepackage{tikz-cd}
\usetikzlibrary{positioning}
\usepackage{graphicx}
\usepackage{float}
\usepackage{cleveref}
\usepackage{geometry}
\usepackage{todonotes}

\usepackage[all]{xy}

\title{A relation on ${(\omega,<)}$ of intermediate degree spectrum on a cone}
\author{Jad Damaj and Matthew Harrison-Trainor\thanks{Harrison-Trainor was partially supported by a Sloan Research Fellowship and by the National Science Foundation under Grant DMS-2153823 / DMS-2419591. This work began while Damaj was an REU student at the University of Michigan funded under this grant DMS-2153823 / DMS-2419591.}}

\newcommand{\vp}{\varphi}
\newcommand{\cB}{\mathcal{B}}

\newcommand{\cL}{\mathcal{L}}

\newcommand{\cA}{\mathcal{A}}
\newcommand{\cR}{\mathcal{R}}

\newcommand{\mc}[1]{\mathcal{#1}}

\newcommand{\Eres}{\upharpoonright\upharpoonright}

\newtheorem{theorem}{Theorem}[section]
\newtheorem{lemma}[theorem]{Lemma}

\newtheorem{conjecture}[theorem]{Conjecture}
\newtheorem{corollary}[theorem]{Corollary}

\newtheorem{claim}{Claim}[theorem]

\theoremstyle{definition}
\newtheorem{definition}[theorem]{Definition}
\newtheorem{example}[theorem]{Example}

\newtheorem{question}[theorem]{Question}

\DeclareMathOperator{\rank}{rank}
\DeclareMathOperator{\maxrank}{maxrank}
\DeclareMathOperator{\minrank}{minrank}

\theoremstyle{remark}
\newtheorem{remark}[theorem]{Remark}

\begin{document}
	
	\maketitle
	
	\begin{abstract}
		We examine the degree spectra of relations on ${(\omega, <)}$. Given an additional relation $R$ on ${(\omega,<)}$, such as the successor relation, the degree spectrum of $R$ is the set of Turing degrees of $R$ in computable copies of ${(\omega,<)}$. It is known that all degree spectra of relations on ${(\omega,<)}$ fall into one of four categories: the computable degree, all of the c.e.\ degrees, all of the $\Delta^0_2$ degrees, or intermediate between the c.e.\ degrees and the $\Delta^0_2$ degrees. Examples of the first three degree spectra are easy to construct and well-known, but until recently it was open whether there is a relation with intermediate degree spectrum on a cone. Bazhenov, Kaloci\'{n}ski, and Wroclawski constructed an example of an intermediate degree spectrum, but their example is unnatural in the sense that it is constructed by diagonalization and thus not canonical, that is, which relation you obtain from their construction depends on which G\"odel encoding (and hence order of enumeration) of the partial computable functions / programs you choose. In this paper, we use the ``on-a-cone'' paradigm to restrict our attention to ``natural'' relations $R$. Our main result is a construction of a natural relation on ${(\omega,<)}$ which has intermediate degree spectrum. This relation has intermediate degree spectrum because of structural reasons.
	\end{abstract}
	
	\section{Introduction}
	
	Let $\cA$ be a mathematical structure such as a group, graph, or linear order. For this paper, we will be solely interested in the case where $\cA$ is the linear order ${(\omega,<)}$ though the definitions can be made in general. Let $R$ be an additional relation on $\cA$ not in the signature of $\cA$. Typical examples are the relation of linear independence on a vector space or the successor relation on a linear order. 
	
	What is the intrinsic complexity of $R$? One way to measure this is to look at how the complexity of $R$ behaves under isomorphisms. In particular, we consider all computable copies $\cB$ of $\cA$ (isomorphic presentations of $\cA$ where all of the functions and relations are computable) and look at the Turing degree of $R$ in $\cB$. The collection of all such Turing degrees is the degree spectrum of $R$. In other words, the degree spectrum measures the possible complexity of $R$ while fixing the complexity of the presentation of the underlying structure.
	
	\begin{definition}
		Given a computable structure $\cA$ and a relation $R$ on $\cA$, we define the degree spectrum of $R$, $\text{DgSp}_{\cA}(R)$, to be the set of Turing degrees   \[ \{ \deg_T (\vp(R)) \colon  \text{$\cB$ is a computable copy of $\cA$ and $\varphi:\cA\cong\cB$}\}\]
		i.e., the images of $R$ in all computable copies of $\cA$ under all isomorphisms.
	\end{definition}
	
	In this paper we follow a series of papers \cite{DMKY,Knoll09,MW18, MHT18,BNW22,BazhenovDariusz,Bazhenov2024} studying the degree spectra of computable relations on the structure ${(\omega,<)}$. The successor relation $S$ plays a particularly important role in this structure. There is one copy of ${(\omega,<)}$, the standard copy, where $S$ is computable. In any other computable copy ${\cL} = (L,\prec)$ we have the successor relation $S^{{\cL}} \subseteq L^2$ on $\mc{A}$. $S^{{\cL}}$ is always a co-c.e.\ set, and hence of c.e.\ degree, and in fact the degree spectrum of $S$ is exactly the c.e.\ degrees. 
	
	In any computable copy $\cL$ there is a unique isomorphism $f_{{\cL}} \colon {\cL} \to {(\omega,<)}$, and the Turing degree of this isomorphism is exactly the Turing degree of $S^{{\cL}}$. Given any other computable relation $R$ on ${(\omega,<)}$, we obtain its image in ${\cL}$ by $R^{{\cL}} = f(R)$, and so ${\cR}^{{\cL}} \leq_T f_{{\cL}} =_T S^{{\cL}} \leq_T \varnothing'$. For many relations $R$, we also always have $R^{{\cL}} \geq_T S^{{\cL}}$, so that the degree spectrum of $R$ is the same as the degree spectrum of $S$, that is, all of the c.e.\ degrees. This is the case, for example, for the double-successor relation. Given $x < y$, we have that $y$ is the successor of $x$ if and only if there is $z > y$ such that $z$ is the double-successor of $x$. Thus the degree spectrum of the double-successor relation---a d.c.e.\ relation---is only the c.e.\ degrees, rather than the d.c.e.\ degrees as one might expect. (A more general version of this argument shows that the degree spectrum of any intrinsically $n$-c.e.\ relation on ${(\omega,<)}$ will be the c.e.\ degrees; see Section 4 of \cite{MHT18}.\footnote{It was incorrectly stated in \cite{MW18} that there are intrinsically $n$-c.e.\ relations on ${(\omega,<)}$ whose degree spectra are exactly the $n$-c.e.\ degrees.})
	
	There are also examples of computable relations on ${(\omega,<)}$ whose degree spectra are only the computable degree (such as an empty relation, or the identity function) and all $\Delta^0_2$ degrees (such as an infinite and co-infinite unary relation). Because ${(\omega,<)}$ is $0'$-categorical, every degree spectrum of a computable relation on ${(\omega,<)}$ is contained within the $\Delta^0_2$ degrees.
	
	Wright showed that for any computable relation $R$ on ${(\omega,<)}$, the degree spectrum is either just the computable degree, or must contain all of the c.e.\ degrees.\footnote{Downey, Khoussainov, Miller, and Yu \cite{DMKY} had previously shown this for unary relations.} Thus no degree spectrum could be intermediate between the computable degree and the c.e.\ degrees. Wright left open the question of which degree spectra intermediate between the c.e.\ degrees and the $\Delta^0_2$ degrees were possible, and in particular, he left open the question of the existence of a degree spectrum strictly intermediate between the c.e.\ degrees and the $\Delta^0_2$ degrees.
	
	In \cite{BNW22}, Bazhenov, Kaloci\'{n}ski, and Wroclawski showed that there is a unary function whose degree spectrum is intermediate. However this relation is unnatural in the sense that it is built via a complicated priority argument and diagonalization. It is not canonical because what relation one obtains by the construction depends on, e.g., the particular choice of G\"odel coding for c.e.\ sets. While computability theory is full of such unnatural counterexamples, such examples are unlikely to show up in the normal course of mathematics.
	
	In this paper we consider degree spectra of \textit{natural} relations on ${(\omega,<)}$. Of course what it means for a relation to be natural is not well-defined and so we use the ``on a cone'' formalism to capture this notion. This formalism originated with Martin's conjecture (see \cite{MontalbanMartin}) and there has been a recent program, first suggested by Montalb\'an, of studying computable structure theory on a cone (e.g., \cite{VaughtMontalban,CsimaHT,MHT18,AndrewsHTSchweber}). Degree spectra on a cone were first studied in the second author's monograph \cite{MHT18}, where relations on ${(\omega,<)}$ were specifically considered.\footnote{Though this monograph appeared in publication slightly before Wright's paper \cite{MW18}, a preprint of Wright's paper was available before the second author started working on \cite{MHT18}.} There the second author asked the on-a-cone version of Wright's question: Is there a relation on ${(\omega,<)}$ whose degree spectrum is intermediate on a cone? At the time an answer to Wright's original question, not on a cone, was still not known. While \cite{BNW22} resolved Wright's question, it did not resolve the on-a-cone version: the degree spectrum of their relation is the c.e.\ degrees on a cone. 
	
	In this paper we resolve the on-a-cone version of Wright's question:

    {
\renewcommand{\thetheorem}{\ref{ex:int-fun}}
\begin{theorem}
  There are relations on ${(\omega,<)}$ whose degree spectrum is strictly between the c.e.\ degrees and the $\Delta^0_2$ degrees on a cone.
\end{theorem}
\addtocounter{theorem}{-1}
}

	In particular, there are computable examples where the cone is the trivial cone, that is, computable relations $R$ on ${(\omega,<)}$ such that relative to any degree $\mathbf{d}$, the degree spectrum of $R$ relative to $\mathbf{d}$ is strictly between the $\mathbf{d}$-c.e.\ degrees and the degrees $\Delta^0_2$ relative to $\mathbf{d}$. Such examples are natural relations of intermediate degree spectrum. To illustrate how natural these examples are, we describe our example. As was the example in \cite{BNW22}, our example will be a unary function $f$. An initial segment of $f$ is as follows:
	
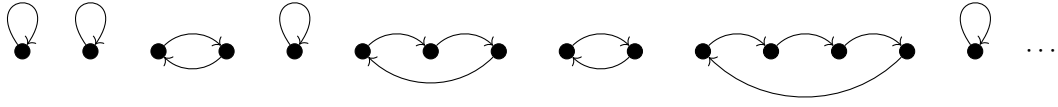
\begin{center}
		\begin{tikzpicture}[scale=0.9,->]
			
			\node[circle,draw=black,fill=black,minimum width=.2cm, minimum height=.2cm, inner sep = 0pt] (box0) at (0,0) {};
			\node[circle,draw=black,fill=black,minimum width=.2cm, minimum height=.2cm, inner sep = 0pt] (box1) at (1,0) {};
			\node[circle,draw=black,fill=black,minimum width=.2cm, minimum height=.2cm, inner sep = 0pt] (box2) at (2,0) {};
			\node[circle,draw=black,fill=black,minimum width=.2cm, minimum height=.2cm, inner sep = 0pt] (box3) at (3,0) {};
			\node[circle,draw=black,fill=black,minimum width=.2cm, minimum height=.2cm, inner sep = 0pt] (box4) at (4,0) {};
			\node[circle,draw=black,fill=black,minimum width=.2cm, minimum height=.2cm, inner sep = 0pt] (box5) at (5,0) {};
			\node[circle,draw=black,fill=black,minimum width=.2cm, minimum height=.2cm, inner sep = 0pt] (box6) at (6,0) {};
			\node[circle,draw=black,fill=black,minimum width=.2cm, minimum height=.2cm, inner sep = 0pt] (box7) at (7,0) {};
			\node[circle,draw=black,fill=black,minimum width=.2cm, minimum height=.2cm, inner sep = 0pt] (box8) at (8,0) {};
			\node[circle,draw=black,fill=black,minimum width=.2cm, minimum height=.2cm, inner sep = 0pt] (box9) at (9,0) {};
			\node[circle,draw=black,fill=black,minimum width=.2cm, minimum height=.2cm, inner sep = 0pt] (box10) at (10,0) {};
			\node[circle,draw=black,fill=black,minimum width=.2cm, minimum height=.2cm, inner sep = 0pt] (box11) at (11,0) {};
			\node[circle,draw=black,fill=black,minimum width=.2cm, minimum height=.2cm, inner sep = 0pt] (box12) at (12,0) {};
			\node[circle,draw=black,fill=black,minimum width=.2cm, minimum height=.2cm, inner sep = 0pt] (box13) at (13,0) {};
			\node[circle,draw=black,fill=black,minimum width=.2cm, minimum height=.2cm, inner sep = 0pt] (box14) at (14,0) {};
			
			\node (box15) at (15,0) {$\cdots$};

			\path (box0) edge[->,loop above,min distance=10mm,in=55,out=125] node {{$$}} (box0);
			\path (box1) edge[->,loop above,min distance=10mm,in=55,out=125] node {{$$}} (box1);
			\path[bend left=45] (box2) edge[->] node {{$$}} (box3);
			\path[bend left=45] (box3) edge[->] node {{$$}} (box2);
			\path (box4) edge[->,loop above,min distance=10mm,in=55,out=125] node {{$$}} (box4);
			\path[bend left=45] (box5) edge[->] node {{$$}} (box6);
			\path[bend left=45] (box6) edge[->] node {{$$}} (box7);
			\path[bend left=45] (box7) edge[->] node {{$$}} (box5);
			\path[bend left=45] (box8) edge[->] node {{$$}} (box9);
			\path[bend left=45] (box9) edge[->] node {{$$}} (box8);
			\path[bend left=45] (box10) edge[->] node {{$$}} (box11);
			\path[bend left=45] (box11) edge[->] node {{$$}} (box12);
			\path[bend left=45] (box12) edge[->] node {{$$}} (box13);
			\path[bend left=45] (box13) edge[->] node {{$$}} (box10);
			\path (box14) edge[->,loop above,min distance=10mm,in=55,out=125] node {{$$}} (box14);

		\end{tikzpicture}
	\end{center}
	One can see that the domain of $f$ is divided up into blocks, and in each block $f$ is a loop. We write $L_n$ for the loop of length $n$, i.e., for the block
	
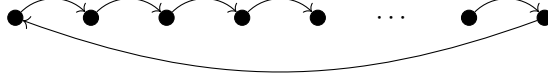
\begin{center}
		\begin{tikzpicture}[->]
			
			\node[circle,draw=black,fill=black,minimum width=.2cm, minimum height=.2cm, inner sep = 0pt] (box0) at (0,0) {};
			\node[circle,draw=black,fill=black,minimum width=.2cm, minimum height=.2cm, inner sep = 0pt] (box1) at (1,0) {};
			\node[circle,draw=black,fill=black,minimum width=.2cm, minimum height=.2cm, inner sep = 0pt] (box2) at (2,0) {};
			\node[circle,draw=black,fill=black,minimum width=.2cm, minimum height=.2cm, inner sep = 0pt] (box3) at (3,0) {};
			\node[circle,draw=black,fill=black,minimum width=.2cm, minimum height=.2cm, inner sep = 0pt] (box4) at (4,0) {};
			\node[] (boxdots) at (5,0) {$\cdots$};
			\node[circle,draw=black,fill=black,minimum width=.2cm, minimum height=.2cm, inner sep = 0pt] (box5) at (6,0) {};
			\node[circle,draw=black,fill=black,minimum width=.2cm, minimum height=.2cm, inner sep = 0pt] (box6) at (7,0) {};

			\path[bend left=45] (box0) edge[->] node {{$$}} (box1);
			\path[bend left=45] (box1) edge[->] node {{$$}} (box2);
			\path[bend left=45] (box2) edge[->] node {{$$}} (box3);
			\path[bend left=45] (box3) edge[->] node {{$$}} (box4);
			\path[bend left=45] (box5) edge[->] node {{$$}} (box6);
			\path[bend left=20] (box6) edge[->] node {{$$}} (box0);
			
		\end{tikzpicture}
	\end{center}
	with $n$ nodes. Then $f$ consists of the following blocks:
	\[ L_1 L_1 L_2 L_1 L_3 L_2 L_4 L_1 L_5 L_2 L_6 L_3 L_7 L_1 L_8 \ldots.\]
	The pattern here is that the blocks in odd positions follow the pattern $L_1 L_2 L_3 L_4 \ldots$ enumerating the natural numbers in increasing order, while the blocks in even positions $L_1 L_1 L_2 L_1 L_2 L_3 \ldots$ are an enumeration of all of the natural numbers such that each number occurs infinitely many times. Thus every block appears infinitely many times, but any pair of blocks appears adjacent to each other at most once.
	
	While we can describe our example simply, the example of \cite{BNW22} does not have a simple description but is actually the result of a complicated priority construction. Moreover, what relation one gets from the priority construction depends on certain non-canonical choices that one makes, such as the choice of G\"odel encoding of the partial computable functions. Finally, their example does not relativize, i.e., the computable relation $R$ produced does not have intermediate degree spectrum relative to $0'$.
	
	When Montalb\'an first suggested studying computable structure theory on a cone, the hope was that one might find more structure in an area of mathematics that generally involved non-structure theorems. Though this is sometimes true (e.g., in \cite{CsimaHT}) it seems to not always be the case; e.g., we know that there are incomparable degree spectra on a cone \cite{MHT18}. The second half of this paper is dedicated to showing that even for relations on ${(\omega,<)}$, one of the simplest non-trivial examples, it seems hard to find structure in the degree spectra on a cone. One of the mains results here is the following theorem:

        {
\renewcommand{\thetheorem}{\ref{thm:example-alpha-even}}
\begin{theorem}
  Fix $\alpha \geq 6$ even. There is a unary function $f$ on ${(\omega,<)}$ whose degree spectrum on a cone contains all of the $\beta$-c.e.\ degrees for $\beta < \alpha$ and does not contain all of the $\alpha$-c.e.\ degrees.
\end{theorem}
\addtocounter{theorem}{-1}
}

\noindent Thus there are uncountably many different degree spectra on a cone. We also prove in \Cref{thm:no-list} that the degree spectra on a cone of these functions are not contained in the $\beta$-c.e.\ degrees for any $\beta$; that is, for any $\beta$, the degree spectrum on a cone contains a non-$\beta$-c.e.\ degree.
	
	\section{Preliminaries}
	
	\subsection{Cones and Martin's measure}
	
	Given a set  $A \subseteq 2^\omega$, we say that $A$ is \textit{degree-invariant} if whenever $X \in A$ and $Y \equiv_T X$, $Y \in A$. If $A$ is degree invariant, we can identify it with the corresponding set of Turing degrees $\{ \deg_T(X) \; \mid \; X \in A\}$.
	
	\begin{definition}
		Given $X \subseteq \omega$, the \textit{cone above $X$} is
		\[ C_X = \{ Y \; \mid \; Y \geq_T X \}.\]
	\end{definition}
	
	As a consequence of Martin's proof of Borel determinacy \cite{BorelDet}, one gets the following theorem.
	
	\begin{theorem}[Martin, \cite{MartinTuring}]\label{thm:cone}
		Every degree-invariant Borel subset of $2^\omega$ either contains a cone or is disjoint from a cone.
	\end{theorem}
	
	\noindent With more determinacy (e.g., analytic determinacy) this can be extended to more complicated degree-invariant sets. For this paper, Borel determinacy will be sufficient.
	
	Thinking of cones as large sets, one can define the $\{0,1\}$-valued Martin's measure on Borel degree-invariant sets by setting $\mu(A) = 1$ if $A$ contains a cone, and $\mu(A) = 0$ if $A$ is disjoint from a cone. Note that the intersection of countably many cones contains a cone, which makes Martin's measure countably additive.
	
	\subsection{Degree spectra on a cone}
	
        We begin by relativizing degree spectra to an oracle.
    
	\begin{definition}
		If $\cA$ is $X$-computable, we define the \textit{degree spectrum of $R$ relative to $X$}, $\text{DgSp}^X_{\cA}(R)$, to be the set of degrees 
		\[\text{DgSp}^X_{\cA}(R) = \{ \deg_T (\vp(R) \oplus X) \colon  \text{$\cB$ is an $X$-computable copy of $\cA$ and $\varphi:\cA\cong\cB$}\}.\]
	\end{definition}

        \noindent There is some discussion in Section 2.3 of \cite{MHT18} of why we use $\deg_T (\vp(R) \oplus X)$ rather than just $\deg_T (\vp(R))$include in this definition.
	
	Let $R$ and $S$ be relations on $\cA$ and $\cB$ respectively. The set of $Y$ such that $\text{DgSp}^Y_{\cA}(R) = \text{DgSp}^Y_{\cB}(S)$ is a degree-invariant Borel set,\footnote{To see this, one must use Scott's theorem \cite{Scott} that for a fixed structure $\cA$, determining whether $\cB \cong \cA$ is Borel.} and so by Theorem \ref{thm:cone}, either contains a cone or is disjoint from a cone. If it contains a cone, then we think of the degree spectra of $R$ and $S$ being equal for most degrees $Y$. Otherwise, there is a cone on which $\text{DgSp}^Y_{\cA}(R) \neq \text{DgSp}^Y_{\cB}(S)$ and we think of their degree spectra as being different for most degrees $Y$.
	
	\begin{definition}
		Let $R$ and $S$ be relations on $\cA$ and $\cB$ respectively. We say that the degree spectrum of $R$ is equal to the degree spectrum of $S$ on a cone if there is some $X$ such that for all $Y \geq_T X$, $\text{DgSp}^Y_{\cA}(R) = \text{DgSp}^Y_{\cB}(S)$. The \textit{degree spectrum of $R$ on a cone} is the equivalence class for $R$ modulo this equivalence relation.
	\end{definition}
	
	We can also define what it means for the degree spectrum of $R$ to be contained in the degree spectrum of $S$ on a cone: if there is some $X$ such that for all $Y \geq_T X$, $\text{DgSp}^Y_{\cA}(R) \subseteq \text{DgSp}^Y_{\cB}(S)$. Sometimes we also want to talk about the contents of the degree spectrum of $R$ on a cone without reference to some other relation, as in the following definition.
	
	\begin{definition}
		Let $R$ be a relation on $\cA$. We say that the degree spectrum of $R$ is equal to the c.e. degrees on a cone if there is some $X$ such that for all $Y \geq_T X$, $\text{DgSp}^Y_{\cA}(R)$ is the set of $Y$-c.e.\ degrees above $Y$. Similarly, we can define what it means for a degree spectrum to be equal to the computable degree or the $\Delta_2^0$ degrees on a cone. 
	\end{definition}

    We also sometimes use the following terminology: $R$ is \textit{intrinsically computable on a cone} if the degree spectrum of $R$ is the computable degree on a cone, and $R$ is \textit{intrinsically of c.e.\ degree on a cone} if the degree spectrum of $R$ is contained within the c.e.\ degrees on a cone. These are just the on-a-cone version of being intrinsically computable or intrinsically of c.e.\ degree.
	
    Note that throughout this paper, the relations and structures need not be computable. This is because for any structure and relation there is a cone on which that structure and relation are computable. Thus we may essentially behave as if all structures and relations are computable.
	
	\subsection{Notation and Definitions}
	
	Recall that Wright \cite{MW18} showed that given a computable relation $R$ on ${(\omega,<)}$, either the degree spectrum of $R$ is the computable degree $\{\mathbf{0}\}$ or the degree spectrum of $R$ contains all of the c.e.\ degrees. For specific types of relations, more is known. For example, Wright showed that if $R$ is unary, then the degree spectrum is either the computable degree $\{\mathbf{0}\}$ or all the $\Delta^0_2$ degrees.
	
	In the process of showing that there is a relation of intermediate degree spectrum, Bazhenov, Kaloci\'{n}ski, and Wroclawski \cite{BNW22} introduced the following class of relatively simple unary functions on ${(\omega,<)}$.
	
	\begin{definition}
		We say that a function $f: {(\omega,<)} \to {(\omega,<)}$ is a \textit{block function} if for each $n$ there is some interval $[a, b]$, containing $n$, that is closed under $f$ and $f^{-1}$. We call the minimal such interval the $f$-block, or simply the block if $f$ is understood, containing $n$. Distinct $f$-blocks cannot have any overlap, and so the domain of $f$ can be divided into blocks.
	\end{definition}
	
	\noindent Note that the function need not be a cycle on a block. For example, the following forms a block:
    
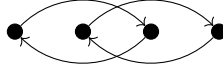
\begin{center}
		\begin{tikzpicture}[scale=0.9,->]
			\node[circle,draw=black,fill=black,minimum width=.2cm, minimum height=.2cm, inner sep = 0pt] (box0) at (0,0) {};
			\node[circle,draw=black,fill=black,minimum width=.2cm, minimum height=.2cm, inner sep = 0pt] (box1) at (1,0) {};
			\node[circle,draw=black,fill=black,minimum width=.2cm, minimum height=.2cm, inner sep = 0pt] (box2) at (2,0) {};
			\node[circle,draw=black,fill=black,minimum width=.2cm, minimum height=.2cm, inner sep = 0pt] (box3) at (3,0) {};
            
            \path[bend left=45] (box0) edge[->] node {{$$}} (box2);
			\path[bend left=45] (box1) edge[->] node {{$$}} (box3);
			\path[bend left=45] (box3) edge[->] node {{$$}} (box1);
            \path[bend left=45] (box2) edge[->] node {{$$}} (box0);
					
		\end{tikzpicture}
	\end{center}
    The example constructed by Bazhenov, Kaloci\'{n}ski, and Wroclawski, as well as the examples that we construct in this paper, will be block functions. 
	
	Each block is finite, so we can list out all of the possible isomorphism types of blocks. We call these isomorphism types the \textit{block types} or simply \textit{types}. We say that $f$ \textit{realises} or \textit{has} a block type if that block type is the type of an $f$-block. We denote the $n$th block type in this enumeration by $I_n$ and will use this to provide an alternate way to represent an arbitrary block function.
	
	\begin{definition}
		Given a block function $f$, we define a sequence $\alpha_f$ of natural numbers, called the string corresponding to $f$, by $\alpha_f(i) = n$ where $I_n$ is the type of the $i$th block that occurs in $f$. 
	\end{definition}
	
	\noindent The string $\alpha_f$ completely determines $f$.
	
	We note that the example of Bazhenov, Kaloci\'{n}ski, and Wroclawski of an intermediate degree spectrum relies on not only the string $\alpha_f$ corresponding to $f$ being non-computable, but the \textit{counting function} $c_f (n) = \#\{i : \alpha_f (i) = n\}$ being non-computable. In our examples, both $\alpha_f$ and $c_f$ will be computable together with any other reasonable information that one might ask for.
	
	Given two $f$-blocks, we say that one embeds into the other if there is an order-preserving and $f$-preserving map from the one block to the other. Given block types $I$ and $J$, we say that $I$ embeds into $J$ if there is an order-preserving and $f$-preserving map from any block of type $I$ to any block of type $J$.
	
	\medskip
	
	Frequently in this paper we will build a computable linear order $(\mc{A},<_\mc{A}) \cong {(\omega,<)}$ (possibly relative to some oracle) by stages. At each stage, we will define a finite linear order $\mc{A}_s$ whose domain is a subset of $\omega$. At each stage $s+1$, we may add new elements to $\mc{A}_s$ to obtain $\mc{A}_{s+1}$, saying where they are in relation to all of the elements of $\mc{A}_s$, so that $\mc{A}_{0} \subseteq \mc{A}_{1} \subseteq \mc{A}_{2} \subseteq \cdots$. We build $\mc{A} = \bigcup_{s} \mc{A}_{s}$.
	
	For clarity, we will always denote by $A = \{a_0,a_1,a_2,\ldots\}$ the domain of $\mc{A}$ to distinguish it from the particular isomorphism type ${(\omega,<)}$. At each stage $s$, we have a guess at an isomorphism $\mc{A} \to {(\omega,<)}$: $0$ corresponds to the first element of $\mc{A}_s$, $1$ to the second element, and so on. We write $\pi_s \colon \mc{A}_s \to {(\omega,<)}$ for the guess at stage $s$ at this partial isomorphism: Given $a \in \mc{A}_s$, we write $\pi_s(a)$ for the corresponding element of ${(\omega,<)}$, that is, $\pi_s(a) = n$ where $n$ is the number of elements less than $a$ in $\mc{A}_s$.\footnote{Distinguishing between elements $a$ of $\mc{A}$ and the elements $\pi_s(a)$ they correspond to in ${(\omega,<)}$ at each stage $s$, a correspondence which will change between stages, will be the primary notational difficulty in this paper and we apologize to the reader for how hard it is to write down any notation for this in an easily readable form. We suggest that if the reader is confused at any point in reading this paper, they should begin by double-checking which entities are in $(\omega,<)$, and which are in the other copy $\mc{A} \cong (\omega,<)$ being considered.}  For a tuple $\bar{a} = (a_1,\ldots,a_n)$, we write $\pi_s(\bar{a})$ for the corresponding tuple $(\pi_s(a_1),\ldots,\pi_s(a_n))$. We also write $\pi$ for the isomorphism $\pi \colon \mc{A} \to {(\omega,<)}$, and so write $\pi(a) = n$ for the corresponding element of ${(\omega,<)}$ at the end of the construction. Think of $\pi_s(a)$ as a stage-by-stage approximation of the final value $\pi(a)$. At each stage $s+1$, new elements may be added to $\mc{A}$. If they are added above $a \in \mc{A}_s$, then $\pi_{s+1}(a) = \pi_s(a)$. But if they are added below $a \in \mc{A}_s$, then $\pi_{s+1}(a) > \pi_s(a)$.
	
	
	While in some constructions we will build a copy $\mc{A}$ of ${(\omega,<)}$ as described above, at other times we will be diagonalizing against all computable copies of ${(\omega,<)}$. Let $(\mc{A}_e)_{e \in \omega}$ be a list of the computable orderings, or more precisely the (possibly partial) $<$-structures $\mc{A}_e$ computed by the $e$th partial computable function. Generally if $\mc{A}_e$ is not a linear order we recognize this at some finite stage and so can ignore it, and since we are only concerned with linear orders isomorphic to ${(\omega,<)}$ we can effectively assume that $\mc{A}_e$ is infinite. So, morally speaking, we can think of each $\mc{A}_e$ as being an infinite computable linear order. At each stage $s$, we have computed finite linear order $\mc{A}_{e,s}$. We can make definitions, like $\pi_{e,s}(a)$, similar to the definitions above. We will usually omit the $e$, writing $\pi_s(a)$ when the structure $\mc{A}_e$ is understood. Note that if $\mc{A}_e$ is not actually isomorphic to ${(\omega,<)}$ then it is possible that $\lim_{s \to \infty} \pi_s(a) = \infty$ as infinitely many elements are added to $\mc{A}_s$ below $a$; but if $\mc{A}_e$ really is isomorphic to ${(\omega,<)}$ then $\lim_{s \to \infty} \pi_s(a)$ will exist and be the isomorphic image of $a$ in ${(\omega,<)}$.
	
	Given $\bar{a} = (a_1,\ldots,a_\ell)$ and $\bar{b} = (b_1,\ldots,b_\ell)$ in ${(\omega,<)}$, we say that $\bar{a}$ \textit{extends to} $\bar{b}$ if $b_1 \geq a_1$ and, for each $i$, $b_{i+1} - b_i \geq a_{i+1} - a_i$. Then $\bar{a}$ extends to $\bar{b}$ if and only if when $\bar{u} \in \mc{A}_s$, and $\pi_s(\bar{u}) = \bar{a}$, it is possible to add elements to $\mc{A}_s$ to produce $\mc{A}_{s+1}$ with $\pi_{s+1}(\bar{u}) = \bar{b}$. Also, given a tuple $\bar{a} = (a_1,\ldots,a_\ell) \in {(\omega,<)}$ and $n \in \mathbb{Z}$ we write $\bar{a} + n$ for $(a_1 + n,\ldots,a_\ell+n)$.
	
	If $f$ is a unary function on ${(\omega,<)}$, then we write $f^{\mc{A}_s}$ for the guess at stage $s$ at the values of $f$ on the elements of $\mc{A}_s$. We define $f^{\mc{A}_s}(a) = b$ if and only if $f(\pi_s(a)) = \pi_s(b)$. Note that $f^{\mc{A}_s}$ might not be a total function, as there might not be enough elements in $\mc{A}_s$. However if $f$ is a block function, then there are initial segments of ${(\omega,<)}$ which are closed under $f$, and so by making some small adjustments we can assume that $f^{\mc{A}_s} \colon \mc{A}_s \to \mc{A}_s$ is total.
	
	
	
	\section{Intermediate Degree Spectrum}
	
	In this section we will construct an example of a unary function whose degree spectrum on a cone is intermediate between the c.e.\ degrees and the $\Delta^0_2$ degrees. First, we will give a general characterization for when the degree spectrum on a cone of a block function strictly contains the c.e.\ degrees. This is the easiest part, and is essentially adapting Section 3 of \cite{MHT18} to the setting of block functions. Second, we will give a general condition for when the degree spectrum of a block function is all $\Delta^0_2$ degrees. This is where most of the work is done. Finally, we give an example and prove that it meets the first condition but not the second condition and thus strictly contains the c.e.\ degrees and is strictly contained within the $\Delta^0_2$ degrees.
	
	\subsection{Containing a non-c.e.\ degree}
	
	We take from \cite{MHT18} a sufficient condition for the degree spectrum of a relation $R$ on an arbitrary structure to strictly contain the c.e.\ degrees on a cone. Properly re-interpreted in the specific case of a block function on ${(\omega,<)}$, this will give us a condition which is both necessary and sufficient.
	
	\begin{definition}
		Let $R$ be an additional relation on a structure $\mc{A}$. We say that $\bar{a}$ is difference-free (or d-free) over $\bar{c}$ if for any tuple $\bar{b}$ and quantifier-free formula $\varphi(\bar{c},\bar{u},\bar{v})$ true of $\bar{a},\bar{b}$ there are $\bar{a}',\bar{b}'$ satisfying $\varphi(\bar{c},\bar{u},\bar{v})$ such that (1) $R$ restricted to $\bar{c},\bar{a}'$ is not the same as $R$ restricted to $\bar{c},\bar{a}$ and (2) for any existential formula $\psi(\bar{c},\bar{u},\bar{v})$ true of $\bar{a}',\bar{b}'$, there are $\bar{a}''$ and $\bar{b}''$ satisfying $\psi(\bar{c},\bar{u},\bar{v})$ and such that $R$ restricted to $\bar{c},\bar{a}'',\bar{b}''$ is the same as $R$ restricted to $\bar{c},\bar{a},\bar{b}$.\footnote{Note that in \cite{MHT18}, the formula $\varphi$ in the definition was allowed to be existential. The definition we give is equivalent and simpler.}
	\end{definition}

    This is a freeness notion in the style of those appearing in, e.g., Chapters 11 and 16 of \cite{AshKnight} in the proofs of theorems of Ash and Nerode \cite{AshNerode} and Barker \cite{Barker}. The definition is somewhat technical, but we will attempt to give some explanation. The best way to think of d-freeness is in terms of constructing a computable copy $\mc{B} \cong \mc{A}$ while having in mind a $\Delta^0_2$ isomorphism $\mc{B} \to \mc{A}$. Since we are constructing $\mc{B}$ computably, while we can change the isomorphism $\mc{B} \to \mc{A}$, we cannot change our mind about which atomic facts are true of elements of $\mc{B}$. Along the way, we will consider how $R^{\mc{B}}$ changes. Suppose that at some point, we have elements $\bar{x},\bar{y}$ of $\mc{B}$ which our isomorphism currently maps to elements $\bar{c},\bar{a}$ with $\bar{a}$ being d-free over $\bar{c}$. One should think of $\bar{u} \mapsto \bar{c}$ as being unchangeable due to higher-priority requirements. Now at some point later in the construction, we might like to cause $R^{\mc{B}}$ to change on $\bar{x},\bar{y}$. Since $\mc{B}$ must be computable, we have at this point committed to certain facts about $\mc{B}$, say a quantifier-free $\varphi(\bar{x},\bar{y},\bar{z})$ where $\bar{z}$ are further elements that we have added to $\mc{B}$ at this point, with the isomorphism mapping $\bar{z} \mapsto \bar{b}$. Note that $\varphi(\bar{c},\bar{a},\bar{b})$ is true in $\mc{A}$. Thus we can find $\bar{a}',\bar{b}'$ as above. Since $\varphi(\bar{c},\bar{a}',\bar{b}')$ is true, we can change our isomorphism to map $\bar{x},\bar{y},\bar{z} \mapsto \bar{c},\bar{a}',\bar{b}'$ as anything that we have committed to about $\bar{x},\bar{y},\bar{z}$ remains true. On the other hand, $R^{\mc{B}}$ has now changed on $\bar{x},\bar{y}$. Then, at some later stage, we might want to undo this change to $R^{\mc{B}}$. At this later point, we may have committed to more facts about $\mc{B}$, say an existential formula $\psi(\bar{x},\bar{y},\bar{z})$ with the new elements we have added to $\mc{B}$ being the witnesses to the existential quantifiers. Then we can find $\bar{a}'',\bar{b}''$ such that $\psi(\bar{c},\bar{a}'',\bar{b}'')$ is true, so that we can change our isomorphism to map $\bar{x},\bar{y},\bar{z} \mapsto \bar{c},\bar{a}'',\bar{b}''$, and furthermore $R^{\mc{B}}$ is the same on $\bar{x},\bar{y},\bar{z}$ as it was before we made any change. This idea, used in a priority argument, is the core of the proof of the following fact. 
	
	\begin{theorem}[Proposition 3.4 of \cite{MHT18}]\label{thm:d-free}
		Let $\cA$ be a structure and $R$ a relation on $\cA$. Suppose that for each tuple $\bar{c}$ there is $\bar{a}$ which is d-free over $\bar{c}$. Then the degree spectrum of $R$ strictly contains the c.e.\ degrees  on a cone.
	\end{theorem}
	
	We do not know whether this condition is also necessary. However if we interpret this for a block function on ${(\omega,<)}$, we get the following condition which is simpler as well as both necessary and sufficient.
	
	\begin{theorem} \label{them:cond-for-c.e.}
		Let $f$ be a block function which is not intrinsically computable on a cone. Then $f$ is intrinsically of c.e. degree on a cone if and only if there are infinitely many blocks that do not embed into a later block. 
	\end{theorem}
	
	\begin{proof}
		First, suppose that there are only finitely many blocks that do not embed into a later block. After some finite initial segment, all blocks that occur embed into some later block. By non-uniformly fixing this initial segment, we can assume that \textit{every} block embeds into another block. We show that for any tuple $\overline{c}$, there is some tuple $\overline{a}$ which is d-free over $\overline{c}$. Given $\overline{c}$, let $\overline{a}$ be some $f$-block of size greater than one such that all its elements are greater than those of $\overline{c}$. Since we assumed that $f$ was not intrinsically computable, it cannot be the identity almost everywhere and so there must be infinitely many blocks of size greater than one. We claim that $\overline{a}$ is d-free over $\overline{c}$.
		
		Now, with $\overline{c}$, $\overline{a}$ as above, suppose there is some quantifier-free formula $\vp(\overline{x}, \overline{u}, \overline{v})$ and tuple $\overline{b}$ such that $\vp(\overline{c}, \overline{a}, \overline{b})$ is true. We make some simplifying assumptions. First, we may assume that there is no repetition amongst the elements of all of the tuples. Second, by including in $\bar{c}$ any elements of $\bar{b}$ which are less than the elements of $\bar{a}$, we may assume that all of the elements of $\bar{b}$ are greater than all of the elements of $\bar{a}$. And third, we can assume that $\bar{b}$ consists of $n$ distinct adjacent blocks $\overline{b} = \overline{b}_{1}\overline{b}_{2} \cdots \overline{b}_{n}$ with $\overline{b}_1$ the first block after $\overline{a}$.
		
		Now define $\overline{a}' = \overline{a} + 1$ and $\overline{b}' = \overline{b} + 1$. Then $\overline{c} \, \overline{a}' \overline{b}'$ has the same order type as $\bar{c},\bar{a},\bar{b}$ and so still satisfies $\vp(\overline{x}, \overline{u}, \overline{v})$. However the values of $f$ on $\overline{a}$ are not the same as $f$ on $\overline{a}'$ as $\overline{a}$ is a block and $\overline{a}'$ is not. Furthermore, suppose $\psi(\overline{x}, \overline{u}, \overline{v})$ is some existential formula satisfied by $\overline{c}$, $\overline{a}'$, $\overline{b}'$. We can again replace this by a quantifier-free formula $\chi(\overline{x}, \overline{u}, \overline{v}, \overline{w})$ which is satisfied by $\overline{c}$, $\overline{a}'$, $\overline{b}', \overline{e}'$ for some tuple $\overline{e}'$. Noting that there are no gaps between the elements of $\bar{a}',\bar{b}'$, we may split $\bar{e}'$ into $\bar{e}_1'\bar{e}_2'$ where the elements of $\bar{e}_1'$ are less than the elements of $\bar{a}'$ and the elements of $\bar{e}_2'$ are greater than the elements of $\bar{a}'\bar{b}'$.
		
		Now let $\overline{a}''$ be the image of $\bar{a}$ in some block into which it embeds, and similarly let $\bar{b}_1'',\ldots,\bar{b}_n''$ be the images of $\bar{b}_1,\ldots,\bar{b}_n$ in blocks into which they embed, choosing these images sufficiently large that $\bar{c} < \bar{e}_1' < \bar{a}'' < \bar{b}_1'' < \bar{b}_2'' < \cdots < \bar{b}_n''$. Finally, choose $\bar{e}_2''$ larger than all of these. Let $\bar{e}'' = \bar{e}_1' \bar{e}_2''$.
		
		Our construction has ensured that $\overline{c} \, \overline{a}'' \overline{b}'' \overline{e}''$ has the same order type as $\overline{c} \, \overline{a}' \overline{b}' \overline{e}'$ and so $\overline{c} \, \overline{a}'' \overline{b}''$ satisfies the formula $\psi(\overline{x}, \overline{u}, \overline{v})$. Moreover, the values of $f$ on $\overline{c} \, \overline{a}'' \overline{b}''$ are the same as the values on $\overline{c} \, \overline{a} \, \overline{b}$. Thus, $\overline{a}$ is d-free over $\overline{c}$, as desired. By Theorem \ref{thm:d-free}, the degree spectrum of $f$ strictly contains the c.e.\ degrees  on a cone.
		
		\medskip
		
		In the other direction, suppose that there are infinitely many $f$-blocks which embed into no larger $f$-blocks. We work on the cone on which we can compute $\alpha_f$ and whether a given block embeds into a later block. We will show that, for any $X$ on this cone, if $f^{\cA}$ is the copy of $f$ in an $X$-computable copy $(\cA, <_{\cA})$ of ${(\omega,<)}$ then $f^{\cA}$ is of $X$-c.e. degree. This implies that the degree spectrum of $R$ on a cone is contained in the c.e.\ degrees. To simplify notation, we will suppress $X$, assuming that $\alpha_f$ is computable and that we can compute whether a given block embeds into any later block.
		
		To show that $f^{\mc{A}}$ is of c.e.\ degree, we will show that for any $k$ it computes the $k$th element of $\mc{A}$. This implies that it computes the successor relation on $\mc{A}$. Recall that the successor relation on $\mc{A}$ is always of c.e.\ degree and computes $f^{\mc{A}}$, so this will imply that $f^{\mc{A}}$ is of c.e.\ degree.
		
		Given $k$, look in ${(\omega,<)}$ for a block,  all of whose elements are greater than $k$,  which does not embed into any greater block. Suppose that this block is the elements $(\ell,\ldots,\ell+m-1)$ of $\omega$, so that there are $\ell$ elements less than this block. Let the block type be $I$. In $\mc{A}$, look for a tuple of elements $a_0 < a_1 < \cdots < a_\ell < \cdots < a_{\ell+m-1}$ of $\mc{A}$ such that $f^{\mc{A}}$ on $(a_\ell,\ldots,a_{\ell+m-1})$ has block type $I$. (This means that $f$ on $(\ell,\ldots,\ell+m-1)$ is the same as $f^{\mc{A}}$ on $(a_\ell,\ldots,a_{\ell+m-1})$.) Once we find such elements, we know that $a_k$ is the $k$th element of $\mc{A}$. This is because block type $I$ does not embed into any later $f^{\mc{A}}$-block (and noting that a block cannot embed into a sequence of blocks without embedding into one of them).
	\end{proof}
	
	\subsection{\texorpdfstring{Not all $\Delta^0_2$ degrees}{Not all Delta02 degrees}}
	
	In this section we give conditions to determine whether the degree spectrum of a block function on a cone consists of all $\Delta^0_2$ degrees or whether the degree spectrum is properly contained in the $\Delta^0_2$ degrees. The key is whether or not an infinite $f$-coding sequence, defined as follows, exists.
	
	\begin{definition}
		Given a block function $f$ on ${(\omega,<)}$, we say that a sequence of intervals $[a_1, b_1], [a_2, b_2], \ldots$ and a collection of maps  $\vp_i \colon [a_i,b_i] \to [a_{i+1},b_{i+1}]$ form an $f$-coding sequence if they satisfy the following conditions 
		\begin{enumerate}
			\item Each interval completely contains all $f$-blocks it intersects, and hence is closed under $f$ and $f^{-1}$.
			\item $\vp_i: [a_i, b_i] \to [a_{i+1}, b_{i+1}]$ are non-decreasing embeddings which preserve the ordering.
			\item $\vp_{i+1} \circ \vp_i \colon [a_i, b_i] \to [a_{i+2},b_{i+2}]$ preserves $f$.
			\item $\vp_1: [a_1, b_1] \to [a_2, b_2]$ does not preserve $f$, i.e., there is some $x \in [a_1,b_1]$ such that $\vp_1(f(x)) \neq f(\vp_1(x))$. Together with the previous property this implies that none of the $\vp_i$ preserves $f$.
			\item $a_{i+1} > b_i$ so that all of the elements of each interval are greater than the elements of the previous interval.
		\end{enumerate}
		For $i < j$ we write $\vp_{i \mapsto j}$ for $\vp_{j-1} \circ \cdots \circ \vp_i \colon [a_i,b_i] \to [a_j,b_j]$. If $i$ and $j$ have the same parity, then $f_{i \mapsto j}$ preserves $f$, and otherwise it does not. Also, the $\vp_{i \mapsto j}$ compose nicely, i.e., $\vp_{j \mapsto k} \circ \vp_{i \mapsto j} = \vp_{i \mapsto k}$.
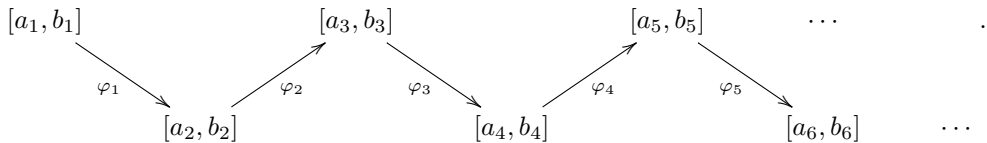
		\[ \xymatrix{ [a_1,b_1]\ar[dr]_{\varphi_1}&& [a_3,b_3]\ar[dr]_{\varphi_{3}} && [a_5,b_5]\ar[dr]_{\varphi_{5}}&\cdots& \\
			&[a_2,b_2]\ar[ur]_{\varphi_{2}}&&[a_4,b_4]\ar[ur]_{\varphi_{4}}&&[a_6,b_6]&\cdots}.\]
        Note that a coding sequences consists of intervals in the standard copy $(\omega,<)$, not in some other $\mc{A} \cong (\omega,<)$.
	\end{definition}
	
	This definition can be motivated as follows. Given a $\Delta_2^0$ set $X$, we want to produce a computable copy $\mc{A}$ of ${(\omega,<)}$ such that $f^{\mc{A}} \equiv_T X$. For each $e$ we want to introduce some elements of $\bar{u}_e$ of $\mc{A}$ coding whether $e \in X$. We will build $\mc{A}$ by stages, with a finite linear order $\mc{A}_s$ at each stage, and $\mc{A}_0 \subseteq \mc{A}_1 \subseteq \mc{A}_2 \subseteq \cdots$ with $\mc{A}$ as the union. At each stage $s$, we guess that $n$th element of $\mc{A}_s$ is really the $n$th element of $\mc{A}$ (hence, the isomorphic image of $n-1 \in {(\omega,<)}$). Thus at stage $s$ we have a guess at the values of $f^{\mc{A}}$ on these elements. When $e$ enters or leaves $X$ at a stage, we must insert elements into $\mc{A}$ below and/or between the elements of $\bar{u}_e$ to change the guess at the values of $f^{\mc{A}}$ on these elements. If, say, $e$ left $X$, so we added new elements to $\mc{A}$ to change the guess at the values of $f^{\mc{A}}$ on $\bar{u}_e$, and then later $e$ re-enters $X$, we must again add new elements to $\mc{A}$ so that the values of $f^{\mc{A}}$ on $\bar{u}_e$ look the same as they did when previously we had $e \in X$. It is important that we cannot just make the values of $f^{\mc{A}}$ on $\bar{u}_e$ different every time $e$ enters or leaves $X$, because then we might be able to compute information about how many times $e$ entered or left $X$.
	
	An $f$-coding sequence is set up exactly to perform such a construction. Given a single $e$, say $e \notin X$ at this stage, we begin with $\bar{u}_e$ being $[a_1,b_1]$. Then, when $e$ enters $X$, we add elements to $\mc{A}$ so that $\bar{u}_e$ becomes $\vp_1([a_1,b_1]) \subseteq [a_2,b_2]$. When $e$ later leaves $X$ we add elements to $\mc{A}$ so that $\bar{u}_e$ becomes $\vp_2(\vp_1([a_1,b_1])) = \varphi_{1 \mapsto 3}([a_1,b_1]) \subseteq [a_3,b_3]$. We continue in this way until $e$ settles. Of course there is much more to deal with, such as all of the other values of $e$, and we have not exactly explained why $X \geq_T f^{\mc{A}}$ as there are many other elements of $\mc{A}$. But this is the basic idea.
	
	Only one infinite coding sequence is necessary to code a $\Delta_2^0$ set as the coding tuples for all values of $e$ can be moved along this same sequence; for this, property (5) is important. Further, we will show that the absence of such a sequence is enough to miss a $\Delta_2^0$ set. This is because, if all coding sequences are finite, then we can produce a $\Delta_2^0$ set which changes the value of each element more times than the length of the coding sequence which attempts to determine its value. This idea will be made more formal in the proofs to follow.
	
	When constructing an $f$-coding sequence, it will be more convenient to omit (5). We call such a sequence of intervals and maps a \textit{weak $f$-coding sequence}. We may occasionally refer to an $f$-coding sequence as a \textit{strong $f$-coding sequence} to emphasize that it is not just a weak $f$-coding sequence. When the function $f$ is clear, we may simply say a coding sequence. In this section, we will be concerned solely with infinite $f$-coding sequences, but in the following section we will be concerned with finite $f$-coding sequences. In those sections we may not always write ``infinite'' or ``finite'' respectively.
	
	By general arguments, one can always turn an infinite weak $f$-coding sequence into an infinite $f$-coding sequence.
	
	\begin{lemma}\label{lem:coding-sequence}
		Let $f$ be a unary function on ${(\omega,<)}$, and suppose that $f$ has an infinite weak $f$-coding sequence. Then $f$ has an infinite  (strong) $f$-coding sequence.
	\end{lemma}
	\begin{proof}
		Suppose that we have a weak $f$-coding sequence with intervals $[a_1, b_1], [a_2, b_2], \ldots$ and maps $\vp_i$. Recall that all of the maps are non-decreasing.
		
		We say that an element $x \in [a_1,b_1]$ \textit{eventually increases to infinity} if, for each $y$, there is some composition of the $\varphi$ which, when applied to $x$, increase it above $y$. That is, if for every $y$ there is some $n$ such that $\varphi_n(\cdots(\varphi_1(x))\cdots) > y$. If $x$ does not eventually increase to infinity, then there must be some $y$ such that for all sufficiently large $n$, $\varphi_n(\cdots(\varphi_1(x))\cdots)  = y$. Note that $b_1$ eventually increases to infinity, as otherwise there would be $n$ such that for all $m \geq n$ the map $\varphi_m$ would be the identity, contradicting (4). Also, we note that if $x$ eventually increases to infinity, then so does every $y \geq x$, and also every $y$ in the same block as $x$.
		
		Let $a_1' \in [a_1,b_1]$ be such that every element of $[a_1,a_1')$ does not eventually increase to infinity, and every element of $[a_1',b_1]$ does eventually increase to infinity. By the above remarks, $[a_1',b_1]$ is closed under blocks. First, we argue that $\varphi_1$ does not preserve $f$ when restricted to $[a_1',b_1]$. This is because it does preserve $f$ when restricted to $[a_1,a_1')$, as we now argue. There is $n$ even such that for all $x \in [a_1,a_1')$
		\[\varphi_n(\cdots(\varphi_1(x))\cdots)  = \varphi_{n+1}(\varphi_n(\cdots(\varphi_1(x))\cdots)).\]
		But $\varphi_{n+1} \circ \cdots \circ \varphi_2 = \varphi_{2 \mapsto n+1}$ preserves $f$ since $n$ is even. So $\varphi_1$ preserves $f$ on $[a_1,a_1')$, and so does not preserve $f$ when restricted to $[a_1',b_1]$.
		
		Now since $a_1'$ eventually increases to infinity, there must be some odd $n$ with $\varphi_n(\cdots(\varphi_1(a_1'))) > b_1$. (Since $b_1$ is the largest element of its block, this also implies that all elements of any block containing the image of an element $[a_1',b_1]$ under $\varphi_n \circ \cdots \circ \varphi_1$ is also greater than $b_1$.) We now replace $[a_1,b_1]$ by $[a_1',b_1]$ and delete all of the blocks $[a_2,b_2],\ldots,[a_n,b_n]$ so that the next block after $[a_1',b_1]$ is $[a_{n+1},b_{n+1}]$. We still have a map $\varphi_1^* = \varphi_{1 \mapsto n+1}$ from $[a_1',b_1]$ to $[a_{n+1},b_{n+1}]$. Now we replace $[a_{n+1},b_{n+1}]$ with the smallest interval $[b_{n+1}',a_{n+1}']$ which contains the image of $[a_1',b_1]$ under $\varphi_1^*$ and which is closed under blocks. By choice of $n$, we get that $a_{n+1}' > b_1$.
		
		Thus we have now obtained
		
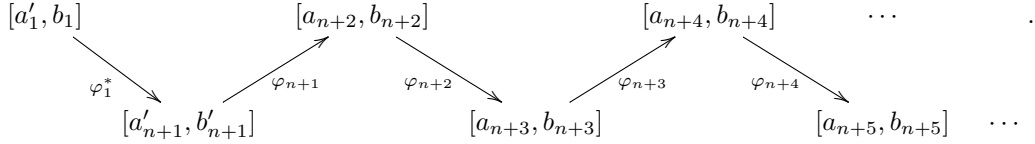
\[ \xymatrix@C=3mm{ [a_1',b_1]\ar[dr]_{\varphi_1^*}&& [a_{n+2},b_{n+2}]\ar[dr]_{\varphi_{n+2}} && [a_{n+4},b_{n+4}]\ar[dr]_{\varphi_{n+4}}&\cdots& \\
			&[a_{n+1}',b_{n+1}']\ar[ur]_{\varphi_{n+1}}&&[a_{n+3},b_{n+3}]\ar[ur]_{\varphi_{n+3}}&&[a_{n+5},b_{n+5}]&\cdots}.\]
		This is a weak $f$-coding sequence with $a_{n+1}' > b_1$. By continuing this process, now with the second interval $[a_{n+1}',b_{n+1}']$, and so on, we eventually obtain an $f$-coding sequence.
	\end{proof}
	
	\begin{theorem} \label{thm:all-d2}
		Let $f$ be a block function such that every block, except for finitely many, embeds into some later block. Then the degree spectrum of $f$, on a cone, is all $\Delta_2^0$ degrees if and only if there is an infinite $f$-coding sequence.
	\end{theorem}

        We prove the two directions of the theorem as separate theorems.

        {
        \renewcommand{\thetheorem}{\ref*{thm:all-d2}A}
        \begin{theorem}\label{thm:A}
            Let $f$ be a block function such that every block, except for finitely many, embeds into some later block. Then if there is an infinite $f$-coding sequence then the degree spectrum of $f$, on a cone, is all $\Delta_2^0$ degrees.
        \end{theorem}
        \addtocounter{theorem}{-1}
        }
	\begin{proof}
		Suppose there is an infinite $f$-coding sequence $[a_1, b_1], [a_2,b_2], [a_3,b_3], \ldots$ with functions $\varphi_i$. We show that, on the cone above $\alpha_f$ and this coding sequence, the degree spectrum of $f$ is all $\Delta_2^0$ degrees. As usual, we will suppress the base of the cone and assume that $\alpha_f$ and this coding sequence are computable. We will also assume, by fixing a finite initial segment, that every block embeds into a later block. Then we must show that the degree spectrum of $f$ is all $\Delta^0_2$ degrees. Since the degree spectrum of $f$ is contained in the $\Delta^0_2$ degrees, we must show that every $\Delta^0_2$ is in the degree spectrum of $f$.
		
		Fix a $\Delta_2^0$ set $X$. We will build a computable copy $\cA$ of ${(\omega, <)}$ such that $f^{\mc{A}} \equiv_T X$. During this construction, the elements we place into the linear order will fall into one of two categories: \textit{coding elements} and \textit{padding elements}. Padding elements are elements that are added to move coding elements, and play no other role in the coding. We will have to ensure that the approximation $f^{\mc{A}_s}$ to $f^{\mc{A}}$ stays the same on the padding elements; we can do this because every block embeds into a later block.  For each $e \in \omega$, there will be some collection of coding elements corresponding to $e$. We divide these into the \textit{initial coding elements} which are the elements first chosen to code whether $e \in X$. It is these elements on which we will check the values of $f^{\mc{A}}$ to determine whether $e$ is in $X$. When we first start coding whether $e \in X$, we will add initial coding elements to $\mc{A}$ corresponding (via the current guess at the isomorphism $\mc{A} \to {(\omega,<)}$) to an interval $[a_i,b_i]$ with $i$ even if $e \notin X$ or odd if $e \in X$. If, at some later stage, we see that $e$ has entered or exited $X$, we insert elements into $\mc{A}$ so that these initial coding elements now correspond to elements of some $[a_j,b_j]$, $j > i$, with the parity of $j$ depending on whether $e \in X$ or $e \notin X$; moreover, the elements in this interval that they correspond to should be the images, under the $\varphi$'s, of the elements they previously corresponded to. It is also possible that some $e' < e$ will enter or exit $X$, causing us to add elements to $\mc{A}$ for the sake of $e'$, and thus requiring us to take action for the coding elements for $e$ as well. Because the embeddings $\varphi$ might not be surjective, as the construction progresses, the initial coding elements will ``gather'' other \textit{subsequent coding elements}. The coding elements for $e$ will be consecutive and we call them the \textit{coding segment} for $e$.
		
		The goal of the construction is to move the coding elements so that the values of $f^{\mc{A}}$ on them mirror whether the corresponding $e$ is in $X$ or not, while also ensuring that once padding elements have been added to the structure, the value of $f^{\mc{A}}$ on those elements does not change. Given $e' < e$, the coding segment for $e$ will be greater than the coding segment for $e'$, so that we can move the segment for $e$ without moving the segment for $e'$.
		
		\medskip
		
		\noindent \textbf{Construction}: 
		For each $e$, starting from stage $e$ where they are defined, we will have the initial coding elements $\bar{u}_e \in \mc{A}$, the full coding segment $\bar{v}_e \in \mc{A}$ containing $\bar{u}_e$, and a restraint $r_e = \max_{\mc{A}} \bar{v}_e \in \mc{A}$ which is the $\mc{A}$-largest element of $\bar{v}$. These are all tuples of $\mc{A}$ (not of $(\omega,<)$, though they do correspond to elements of $(\omega,<)$), and we write each tuple in $\mc{A}$-increasing order. The initial coding elements $\bar{u}_e$ will never change once defined, while the coding segment $\bar{v}_e$ may have more elements added, and the restraint $r_e$ may increase in the $\mc{A}$ order as, when $\bar{v}_e$ obtains more elements, $r_e = \max_{\mc{A}} \bar{v}_e$ may increase in the $\mc{A}$-order. The previous sentence is referring to which elements of $\mc{A}$ these tuples consist of, rather than referring to which elements of $(\omega,<)$ they correspond to as this may increase throughout the construction. So for example, $\bar{u}_e$ is a fixed tuple of $\mc{A}$, but may correspond to increasingly greater tuples from $(\omega,<)$ as more and more elements are added to $\mc{A}$. We will use $\bar{v}_e[s]$ and $r_e[s]$ to refer to these elements at stage $s$. We will always have that if $e' \leq e$ then $\bar{v}_e > r_{e'}$ in the $\mc{A}$-ordering so that the coding segment for $e$ is free of the restraint of $e'$. 
		
		Recall that at each stage $s$, given a tuple $\bar{a} \in \mc{A}_s$, such as the tuples $\bar{u}_e$ and $\bar{v}_e$, $\pi_s(\bar{a})$ is the corresponding tuple in ${(\omega,<)}$. We also have $f^{\mc{A}_s}$ which is the function $f$ on $\mc{A}_s$.
		
		From each stage to the next, we will maintain the following properties:
		\begin{enumerate}
			\item At each stage $s \geq e$, the coding elements $\bar{v}_e$ of $\mc{A}$ correspond to $\pi_s(\bar{v}_e)$ in ${(\omega,<)}$ which is an interval $[a_i,b_i]$, with $i$ odd if $e \in X_s$ or $i$ even if $e \notin X_s$.
			\item If at stage $s \geq e$ the coding elements $\bar{v}_e[s]$ corresponded (via $\pi_s \colon \mc{A}_s \to {(\omega,<)}$) to an interval $[a_i,b_i]$, and at stage $s+1$ the coding elements $\bar{v}[s+1]$ correspond (via $\pi_{s+1} \colon \mc{A}_s \to {(\omega,<)}$) to an interval $[a_j,b_j]$, then the coding elements $\bar{v}_e[s]$ at stage $s$ correspond at stage $s+1$ to the images, under the $\varphi$, of the elements of ${(\omega,<)}$ that they corresponded to at stage $s$. That is,
			\[ \pi_{s+1}(\bar{v}_e[s]) = \varphi_{i \mapsto j} (\pi_s(\bar{v}_e[s])).\]
			\item If, at stage $s$ an element $a \in \mc{A}$ is a padding element, by which we mean that it is not a coding element (i.e., not part of any $\bar{v}_e[s]$), then at stage $s+1$ it is still not a coding element, and $f^{\mc{A}_s}(a) = f^{\mc{A}_{s+1}}(a)$.
			\item From stage $s \geq e$ to stage $s+1$, elements can only be added to $\mc{A}$ below $r_e$ if there is $e' \leq e$ which entered or left $X$ at stage $s+1$, i.e., with $e' \in X_{s} \triangle X_{s+1}$. Thus, if $X_{s+1} \Eres_{e} = X_s \Eres_e$,\footnote{We write $Y \Eres_n$ for the restriction of $Y$ to indices $\leq n$ (with $Y \upharpoonright_n$ being the restriction of $Y$ to indices $< n$).} then the partial isomorphism $\pi_s \colon \mc{A}_s \to {(\omega,<)}$ does not change at stage $s+1$ on elements below $r_e$.
		\end{enumerate}
		
		\smallskip
		
		\noindent\textit{Stage $0$:} We begin with $\mc{A}_0$ being empty.
		
		\smallskip
		
		\noindent\textit{Stage $s$:} At this stage in the construction our partial linear order $\mc{A}_s$ can be partitioned into a finite number of coding segments and padding blocks, each of which form an interval:
		\[ \bar{p}_0 <_{\mc{A}} \bar{v}_0[s] <_{\mc{A}} \bar{p}_1 <_{\mc{A}} \bar{v}_1[s] <_{\mc{A}} \cdots .\]
		We add new elements to $\mc{A}_s$ to code $X_{s+1}$.
		
		We do not do anything to $\bar{p}_0$. We begin by handling $\bar{v}_0[s]$. If $X_{s+1}(0) = X_s(0)$, then we do nothing. But if $X_{s+1}(0) \neq X_s(0)$, then we must act. Suppose that $0 \in X_{s}$ but $0 \notin X_{s+1}$. Then $\bar{v}_0[s]$ corresponds to the elements of an interval $[a_i,b_i]$ for $i$ even. We insert $a_{i+1} - a_i$ new padding elements above the elements of $\bar{p}_0$ and below the elements of $\bar{v}_0[s]$. Also add new coding elements ($b_{i+1} + a_i - a_{i+1} - b_i$ of them) to ensure that $\bar{v}_0[s]$ corresponds, in $\mc{A}_{s+1}$, to $\vp_i(\pi_s(\bar{v}_0[s])) \subseteq [a_{i+1},b_{i+1}]$. These new elements will be subsequent coding elements; let $\bar{v}_0[s+1]$ consist of $\bar{v}_0[s]$ together with these new elements of $\mc{A}$. If $0 \notin X_s$ but $0 \in X_{s+1}$, then we act similarly but move $\bar{v}_0$ from an interval $[a_i,b_i]$ for $i$ odd to $[a_{i+1},b_{i+1}]$ with $i+1$ even.
		
		Now we have added several elements below $\bar{p}_1$. We must ensure that the values of $f$ on $\bar{p}_1$ at stage $s+1$ are the same as they were at stage $s$. For each block making up $\bar{p}_1$, insert new padding elements below the least element of this block and possibly between the elements of this block to move them up to the image of the original $f$-block in some other $f$-block into which it embeds. (We use here the fact that each $f$-block embeds into infinitely many other blocks.) Further, we add enough padding elements on the end to complete these blocks. This ensures that for all padding elements present at stage $s$, the value of $f^{\cA_{s+1}}$ on these elements at stage $s+1$ is the same as it was at stage $s$. 
		
		We continue this process, in turn, with the subsequence coding segments and padding blocks. In increasing order, we ensure that each of these segments from $\cA_{s}$ still satisfy our requirements in $\cA_{s+1}$. However we must now take into account the fact that we may have already inserted elements.
		\begin{itemize}
			\item Given a padding segment $\bar{p}_i$, consider each set of elements which corresponded at stage $s$ to a block. We must make sure that those same elements correspond in $\mc{A}_{s+1}$ to a block of the same type. For each block making up $\bar{p}_i$, insert new padding elements below the least element of this block and possibly between the elements of this block to move them up to the image of the original $f$-block in some other $f$-block into which it embeds. Further, ensure we add enough padding elements on the end to complete this block. This ensures that for all padding elements present at stage $s$, the value of $f^{\cA_s}$ on these elements is the same as at stage $s+1$. 
			\item Given $\bar{v}_e[s]$ the coding segment corresponding to $e$, first check the value of $X_{s+1}(e)$. Even if $X_{s+1}(e) = X_s(e)$ has not changed, we may have added new elements below $\bar{v}_e[s]$ so that it no longer corresponds to the same interval $[a_k,b_k]$ that it did at stage $s$. Let $\varphi_{k \mapsto \ell}$ be such that (i) $\ell$ is even if $e \in X_{s+1}$ and $\ell$ is odd if $e \notin X_{s+1}$, and (ii) $\ell$ is sufficiently large that by adding elements below and between the elements of $\bar{v}_e[s]$ that at stage $s+1$ the segment $\bar{v}_e[s]$ corresponds to the image, under $\varphi_{k \mapsto \ell}$, of $[a_k,b_k] = \pi_s(\bar{v}_e[s])$. (Note that (5) in the definition of a coding sequence---the difference between a weak coding sequence and a coding sequence---is key to obtain such an $[a_\ell,b_\ell]$ with $a_\ell$ large enough to accommodate the new elements that have been added to the linear order below $\bar{v}_e[s]$.) Any new elements that were inserted and end up in the new interval are added to the collection of coding elements $\bar{v}_e[s+1]$ corresponding to $e$, and the other newly added elements are padding elements. Further, we ensure that enough new coding elements are after the final element to ensure we complete the entire interval, putting these elements in $\bar{v}_e[s+1]$ as well. Thus $\bar{v}_e[s+1]$ will now consist of $\pi_{s+1}^{-1}([a_\ell,b_\ell])$.
			\item Finally, we introduce the coding elements corresponding to $s$. After we have ensured all previously added elements still satisfy the requirements check the value of $X_{s+1}(s)$ (which we may assume is $0$, i.e., $s \notin S_{x+1}$) and identify the next interval of the form $[a_i, b_i]$ such that $i$ is even if $s \in X_{s+1}$ and $i$ is odd if $s \notin X_{s+1}$, and with $a_i$ greater than the length of the linear order at this stage. Insert enough new padding elements to the end of the linear order to extend it to have length $a_i$, then add $b_i-a_i+1$ new coding elements corresponding to the interval $[a_i,b_i]$. These are the initial coding elements $\bar{u}_s[s+1]$ for $s$, and make up the coding segment $\bar{v}_s[s+1]$ at this stage. This ensures we have satisfied requirement $R_s$ at stage $s$. 
		\end{itemize}
		\textbf{Verification}: First, to see that $\cA$ is really a computable copy of ${(\omega, <)}$, observe that for any fixed $a \in \mc{A}$, only finitely many elements are inserted below $a$. This follows from (4). Because $a$ was added at stage $s$ in the construction, then there can be at most $s-1$ coding blocks below it. An element is only inserted below $a$ if some $e$ corresponding to one of these coding blocks enters or leaves $X$, with $a$ below the restraint. After these $s-1$ elements of the $\Delta_2^0$ set $X$ stop changing values, elements will stop being inserted below $a$. Hence, since these are finitely many elements of a $\Delta_2^0$ set, this will occur in a finite number of stages.
		
		To show that $f^{\cA} \equiv_T X$, we use the following claim which follows from (2).
		
		\begin{claim}\label{lem:comp-of-maps}
			Suppose that $s < t$. If, at stage $s$, $\bar{v}_e[s]$ corresponds to $[a_m,b_m]$ and, at stage $t$, $\bar{v}_e[t]$ corresponds to $[a_n,b_n]$, then the elements of $\bar{v}_e[s]$ corresponds at stage $t$ to the images under $\varphi_{m \mapsto n}$ of the elements they corresponded to at stage $s$:
			\[ [a_n,b_n] \supseteq \pi_{t}(\bar{v}_e[s]) = \varphi_{m \mapsto n} (\pi_s(\bar{v}_e[s])) = \varphi_{m \mapsto n} [a_m,b_m]. \]
			In particular, if $m$ and $n$ have the same parity, then $f^{\mc{A}_t}(a) = f^{\mc{A}_s}(a)$. If $m$ and $n$ have opposite parities, then there is $a \in \bar{u}_e$ such that $f^{\mc{A}_t}(a) \neq f^{\mc{A}_s}(a)$.
		\end{claim}
		\begin{proof}
			The first part follows from (2) using induction. The second part follows from the definition of a coding sequence, particularly which $\varphi_{m \mapsto n}$ are $f$-preserving.
		\end{proof}
		
		We first use the lemma and (1) to show that $f^{\cA} \ge_T X$. Given some element $e$ run the above construction, which is computable, until stage $e+1$ when the initial coding elements $\bar{u}_e$ corresponding to $e$ are introduced. Now, compute the true value of $f^{\mc{A}}$ on these elements. If $f^{\mc{A}}$ is the same on $\bar{u}_e$ as it was at stage $e+1$, then $X(e) = X_e(e+1)$. Otherwise, if $f^{\mc{A}}$ is different on $\bar{u}_e$ than it was at stage $e+1$, then $X(e) = 1- X_e(e+1)$.
		
		Finally, we show that $X \ge_T f^{\cA}$. This uses the lemma as well as properties (1) and (3). Given some element $a \in \mc{A}$ run the above construction until $a$ is added to the linear order. Using (3) if $a$ is added as a padding element then the construction ensures that $f^{\cA}(a)$ does not change and so we take the value at this stage. If $a$ is a coding element corresponding to $e$ then we use (1) and the lemma. We will see that $f^{\cA}(a)$ takes on one of two values depending on $X(e)$. Say that $a$ first appears in $\mc{A}$ at a stage $s$ in a coding block $\bar{v}_e[s]$, corresponding at that stage to $[a_i,b_i]$ and with $X_s(e) = k$. If $X(e) = X_s(e) = k$, then $f^{\cA}(a)$ is the same as it was at stage $s$. Otherwise, suppose that $X(e) \neq X_s(e) = k$. Then, at some stage $t > s$, we find that $X(e) = X_t(e) = 1-k$. At this stage, $a$ is an element of $\bar{v}_e[t]$ which corresponds to $[a_j,b_j]$ with $j$ of a different parity from $i$. Then $f^{\cA}(a)$ is the same as it was at stage $t$.
    \end{proof}

    {
    \renewcommand{\thetheorem}{\ref*{thm:all-d2}B}
    \begin{theorem}\label{thm:B}
        Let $f$ be a block function such that every block, except for finitely many, embeds into some later block. Then if the degree spectrum of $f$ on a cone is all $\Delta_2^0$ degrees then there is an infinite $f$-coding sequence.
    \end{theorem}
    \addtocounter{theorem}{-1}
    }
    \begin{proof}
		We work on the cone above $\alpha_f$ (which we suppress) and show that if no infinite $f$-coding sequences exist then we can produce $\Delta^0_2$ set $C$ such that there is no computable copy $\cL$ of ${(\omega,<)}$ with $f^{\cL} \equiv_T C$. To do this, we meet the  following requirements
		\[ R_{e, i, j}: \text{If $\cL_e \cong {(\omega,<)}$, then either $\Phi_i^{f^{\cL_e}}\neq C$ or $\Psi_j^C \neq f^{\cL_e}$}\]
		where $\cL_e$ is a computable listing of the (possibly partial) linear orders. The construction is a finite injury priority construction. Given a set $X$, we make use of the notation $X[0,\ldots,k]$ for the sequence $\langle X(0),\ldots,X(k) \rangle$.
		
		The strategy for $R_{e, i, j}$ is as follows. First, if $e$th program fails to code a linear order the requirement is automatically satisfied and so we will assume that at all stages $s$, $\cL_{e,s}$ is a linear order. When we say that a requirement $\mc{R}$ places a restraint on $C$, this means that it is restraining that part of $C$ from changing for the sake of a lower-priority requirement. $\mc{R}$ itself may cause $C$ to change on the restraint.
		\begin{enumerate}
			\item To initialize this requirement choose some $x$ that has not yet been restrained, restrain it, and assign it to this requirement. (This requirement will not cause $C$ to change on any value other than $x$, and so this is the only function of the restraints placed by the higher-priority requirements.) We begin with $C(x) = 0$. We say the strategy is in Phase 0. 
			\item At stage $s$, if the requirement is in Phase 0, say this requirement requires attention if there are computations 
			\[ \Phi_{i, s}^{f^{\cL_{e, s}}}(x) = 0 = C_s(x) \quad \text{with use } u_0\]
			and 
			\[ \Psi_{j, s}^{C_s}[0, \ldots, m_0] = f^{\cL_{e, s}}[0, \ldots, m_0] \quad \text{with use } v_0 \]
			where $m_0 \ge u_0$ is such that all $f$-blocks that intersect $[0, \ldots, u_0]$ in $\mc{A}_{e,s}$, and all elements less than them in $\cL_{e,s}$, are completely contained in $[0, \ldots, m_0]$.
			
			When this requirement acts, restrain $[0, \ldots, v_0]$ in $C$ and define $C_{s+1}(x) =1$. Finally, move the requirement to Phase 1. 

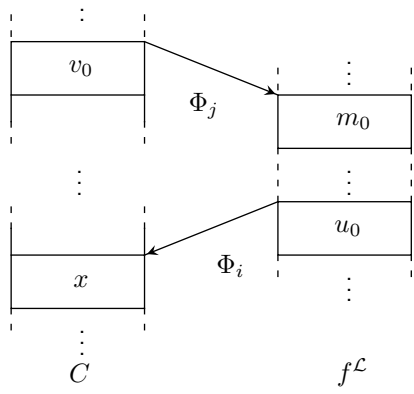
\[\begin{pgfpicture}
\pgfpathmoveto{\pgfpoint{99.75bp}{444.169061bp}}
\pgfpathlineto{\pgfpoint{250.25bp}{591.645623bp}}
\pgfusepath{use as bounding box}
\begin{pgfscope}
\pgfsetlinewidth{0.5bp}
\pgfsetrectcap
\pgfsetmiterjoin
\pgfsetmiterlimit{10.0}
\pgfpathmoveto{\pgfpoint{100.0bp}{491.889764bp}}
\pgfpathlineto{\pgfpoint{100.0bp}{471.889764bp}}
\pgfpathlineto{\pgfpoint{150.0bp}{471.889764bp}}
\pgfpathlineto{\pgfpoint{150.0bp}{491.889764bp}}
\pgfpathlineto{\pgfpoint{100.0bp}{491.889764bp}}
\pgfclosepath
\definecolor{strokepaint}{rgb}{0.0,0.0,0.0}\pgfsetstrokecolor{strokepaint}
\pgfusepath{stroke}
\end{pgfscope}
\begin{pgfscope}
\pgftransformcm{1.0}{-0.0}{0.0}{1.0}{\pgfpoint{126.167969bp}{481.954217bp}}
\pgftext[center,center]{\sffamily\mdseries\upshape\normalsize\color[rgb]{0.0,0.0,0.0}$x$}
\end{pgfscope}
\begin{pgfscope}
\pgftransformcm{1.0}{-0.0}{0.0}{1.0}{\pgfpoint{125.679688bp}{447.872186bp}}
\pgftext[center,center]{\sffamily\mdseries\upshape\normalsize\color[rgb]{0.0,0.0,0.0}$C$}
\end{pgfscope}
\begin{pgfscope}
\pgftransformcm{1.0}{-0.0}{0.0}{1.0}{\pgfpoint{228.019531bp}{447.868279bp}}
\pgftext[center,center]{\sffamily\mdseries\upshape\normalsize\color[rgb]{0.0,0.0,0.0}$f^{\mathcal{L}}$}
\end{pgfscope}
\begin{pgfscope}
\pgfsetlinewidth{0.5bp}
\pgfsetrectcap
\pgfsetmiterjoin
\pgfsetmiterlimit{10.0}
\pgfpathmoveto{\pgfpoint{100.0bp}{571.889764bp}}
\pgfpathlineto{\pgfpoint{100.0bp}{551.889764bp}}
\pgfpathlineto{\pgfpoint{150.0bp}{551.889764bp}}
\pgfpathlineto{\pgfpoint{150.0bp}{571.889764bp}}
\pgfpathlineto{\pgfpoint{100.0bp}{571.889764bp}}
\pgfclosepath
\definecolor{strokepaint}{rgb}{0.0,0.0,0.0}\pgfsetstrokecolor{strokepaint}
\pgfusepath{stroke}
\end{pgfscope}
\begin{pgfscope}
\pgftransformcm{1.0}{-0.0}{0.0}{1.0}{\pgfpoint{126.212891bp}{561.954217bp}}
\pgftext[center,center]{\sffamily\mdseries\upshape\normalsize\color[rgb]{0.0,0.0,0.0}$v_0$}
\end{pgfscope}
\begin{pgfscope}
\pgfsetlinewidth{0.5bp}
\pgfsetrectcap
\pgfsetmiterjoin
\pgfsetmiterlimit{10.0}
\pgfpathmoveto{\pgfpoint{200.0bp}{511.889764bp}}
\pgfpathlineto{\pgfpoint{200.0bp}{491.889764bp}}
\pgfpathlineto{\pgfpoint{250.0bp}{491.889764bp}}
\pgfpathlineto{\pgfpoint{250.0bp}{511.889764bp}}
\pgfpathlineto{\pgfpoint{200.0bp}{511.889764bp}}
\pgfclosepath
\definecolor{strokepaint}{rgb}{0.0,0.0,0.0}\pgfsetstrokecolor{strokepaint}
\pgfusepath{stroke}
\end{pgfscope}
\begin{pgfscope}
\pgftransformcm{1.0}{-0.0}{0.0}{1.0}{\pgfpoint{226.320312bp}{501.954217bp}}
\pgftext[center,center]{\sffamily\mdseries\upshape\normalsize\color[rgb]{0.0,0.0,0.0}$u_0$}
\end{pgfscope}
\begin{pgfscope}
\pgfsetlinewidth{0.5bp}
\pgfsetrectcap
\pgfsetmiterjoin
\pgfsetmiterlimit{10.0}
\pgfpathmoveto{\pgfpoint{200.0bp}{551.889764bp}}
\pgfpathlineto{\pgfpoint{200.0bp}{531.889764bp}}
\pgfpathlineto{\pgfpoint{250.0bp}{531.889764bp}}
\pgfpathlineto{\pgfpoint{250.0bp}{551.889764bp}}
\pgfpathlineto{\pgfpoint{200.0bp}{551.889764bp}}
\pgfclosepath
\definecolor{strokepaint}{rgb}{0.0,0.0,0.0}\pgfsetstrokecolor{strokepaint}
\pgfusepath{stroke}
\end{pgfscope}
\begin{pgfscope}
\pgftransformcm{1.0}{-0.0}{0.0}{1.0}{\pgfpoint{228.490234bp}{541.954217bp}}
\pgftext[center,center]{\sffamily\mdseries\upshape\normalsize\color[rgb]{0.0,0.0,0.0}$m_0$}
\end{pgfscope}
\begin{pgfscope}
\pgftransformcm{1.0}{-0.0}{0.0}{1.0}{\pgfpoint{226.664062bp}{562.950311bp}}
\pgftext[center,center]{\sffamily\mdseries\upshape\normalsize\color[rgb]{0.0,0.0,0.0}$\vdots$}
\end{pgfscope}
\begin{pgfscope}
\pgfsetlinewidth{0.5bp}
\pgfsetrectcap
\pgfsetmiterjoin
\pgfsetmiterlimit{10.0}
\pgfsetdash{{2.0bp}{4.0bp}}{0.0bp}
\pgfpathmoveto{\pgfpoint{200.0bp}{561.889764bp}}
\pgfpathlineto{\pgfpoint{200.0bp}{551.889764bp}}
\definecolor{strokepaint}{rgb}{0.0,0.0,0.0}\pgfsetstrokecolor{strokepaint}
\pgfusepath{stroke}
\end{pgfscope}
\begin{pgfscope}
\pgfsetlinewidth{0.5bp}
\pgfsetrectcap
\pgfsetmiterjoin
\pgfsetmiterlimit{10.0}
\pgfsetdash{{2.0bp}{4.0bp}}{0.0bp}
\pgfpathmoveto{\pgfpoint{250.0bp}{561.889764bp}}
\pgfpathlineto{\pgfpoint{250.0bp}{551.889764bp}}
\definecolor{strokepaint}{rgb}{0.0,0.0,0.0}\pgfsetstrokecolor{strokepaint}
\pgfusepath{stroke}
\end{pgfscope}
\begin{pgfscope}
\pgfsetlinewidth{0.5bp}
\pgfsetrectcap
\pgfsetmiterjoin
\pgfsetmiterlimit{10.0}
\pgfpathmoveto{\pgfpoint{200.0bp}{511.889764bp}}
\pgfpathlineto{\pgfpoint{150.0bp}{491.889764bp}}
\definecolor{strokepaint}{rgb}{0.0,0.0,0.0}\pgfsetstrokecolor{strokepaint}
\pgfusepath{stroke}
\end{pgfscope}
{\begin{pgfscope}
\definecolor{fillpaint}{rgb}{0.0,0.0,0.0}\pgfsetfillcolor{fillpaint}
\pgfpathqmoveto{150.928482bp}{492.261163bp}
\pgfpathqlineto{153.899597bp}{495.603661bp}
\pgfpathqlineto{153.249664bp}{493.18963bp}
\pgfpathqlineto{155.385162bp}{491.889764bp}
\pgfclosepath
\pgfusepathqfill
\end{pgfscope}}
\begin{pgfscope}
\pgfsetlinewidth{0.5bp}
\pgfsetrectcap
\pgfsetmiterjoin
\pgfsetmiterlimit{10.0}
\pgfpathmoveto{\pgfpoint{150.0bp}{571.889764bp}}
\pgfpathlineto{\pgfpoint{200.0bp}{551.889764bp}}
\definecolor{strokepaint}{rgb}{0.0,0.0,0.0}\pgfsetstrokecolor{strokepaint}
\pgfusepath{stroke}
\end{pgfscope}
{\begin{pgfscope}
\definecolor{fillpaint}{rgb}{0.0,0.0,0.0}\pgfsetfillcolor{fillpaint}
\pgfpathqmoveto{199.071518bp}{552.261163bp}
\pgfpathqlineto{194.614838bp}{551.889764bp}
\pgfpathqlineto{196.750336bp}{553.18963bp}
\pgfpathqlineto{196.100403bp}{555.603661bp}
\pgfclosepath
\pgfusepathqfill
\end{pgfscope}}
\begin{pgfscope}
\pgfsetlinewidth{0.5bp}
\pgfsetrectcap
\pgfsetmiterjoin
\pgfsetmiterlimit{10.0}
\pgfpathmoveto{\pgfpoint{100.0bp}{551.889764bp}}
\pgfpathlineto{\pgfpoint{100.0bp}{541.889764bp}}
\definecolor{strokepaint}{rgb}{0.0,0.0,0.0}\pgfsetstrokecolor{strokepaint}
\pgfusepath{stroke}
\end{pgfscope}
\begin{pgfscope}
\pgfsetlinewidth{0.5bp}
\pgfsetrectcap
\pgfsetmiterjoin
\pgfsetmiterlimit{10.0}
\pgfpathmoveto{\pgfpoint{150.0bp}{551.889764bp}}
\pgfpathlineto{\pgfpoint{150.0bp}{541.889764bp}}
\definecolor{strokepaint}{rgb}{0.0,0.0,0.0}\pgfsetstrokecolor{strokepaint}
\pgfusepath{stroke}
\end{pgfscope}
\begin{pgfscope}
\pgfsetlinewidth{0.5bp}
\pgfsetrectcap
\pgfsetmiterjoin
\pgfsetmiterlimit{10.0}
\pgfsetdash{{2.0bp}{4.0bp}}{0.0bp}
\pgfpathmoveto{\pgfpoint{150.0bp}{541.889764bp}}
\pgfpathlineto{\pgfpoint{150.0bp}{531.889764bp}}
\definecolor{strokepaint}{rgb}{0.0,0.0,0.0}\pgfsetstrokecolor{strokepaint}
\pgfusepath{stroke}
\end{pgfscope}
\begin{pgfscope}
\pgfsetlinewidth{0.5bp}
\pgfsetrectcap
\pgfsetmiterjoin
\pgfsetmiterlimit{10.0}
\pgfsetdash{{2.0bp}{4.0bp}}{0.0bp}
\pgfpathmoveto{\pgfpoint{100.0bp}{541.889764bp}}
\pgfpathlineto{\pgfpoint{100.0bp}{531.889764bp}}
\definecolor{strokepaint}{rgb}{0.0,0.0,0.0}\pgfsetstrokecolor{strokepaint}
\pgfusepath{stroke}
\end{pgfscope}
\begin{pgfscope}
\pgfsetlinewidth{0.5bp}
\pgfsetrectcap
\pgfsetmiterjoin
\pgfsetmiterlimit{10.0}
\pgfpathmoveto{\pgfpoint{150.0bp}{491.889764bp}}
\pgfpathlineto{\pgfpoint{150.0bp}{501.889764bp}}
\definecolor{strokepaint}{rgb}{0.0,0.0,0.0}\pgfsetstrokecolor{strokepaint}
\pgfusepath{stroke}
\end{pgfscope}
\begin{pgfscope}
\pgfsetlinewidth{0.5bp}
\pgfsetrectcap
\pgfsetmiterjoin
\pgfsetmiterlimit{10.0}
\pgfsetdash{{2.0bp}{4.0bp}}{0.0bp}
\pgfpathmoveto{\pgfpoint{150.0bp}{501.889764bp}}
\pgfpathlineto{\pgfpoint{150.0bp}{511.889764bp}}
\definecolor{strokepaint}{rgb}{0.0,0.0,0.0}\pgfsetstrokecolor{strokepaint}
\pgfusepath{stroke}
\end{pgfscope}
\begin{pgfscope}
\pgfsetlinewidth{0.5bp}
\pgfsetrectcap
\pgfsetmiterjoin
\pgfsetmiterlimit{10.0}
\pgfpathmoveto{\pgfpoint{100.0bp}{501.889764bp}}
\pgfpathlineto{\pgfpoint{100.0bp}{491.889764bp}}
\definecolor{strokepaint}{rgb}{0.0,0.0,0.0}\pgfsetstrokecolor{strokepaint}
\pgfusepath{stroke}
\end{pgfscope}
\begin{pgfscope}
\pgfsetlinewidth{0.5bp}
\pgfsetrectcap
\pgfsetmiterjoin
\pgfsetmiterlimit{10.0}
\pgfsetdash{{2.0bp}{4.0bp}}{0.0bp}
\pgfpathmoveto{\pgfpoint{100.0bp}{501.889764bp}}
\pgfpathlineto{\pgfpoint{100.0bp}{511.889764bp}}
\definecolor{strokepaint}{rgb}{0.0,0.0,0.0}\pgfsetstrokecolor{strokepaint}
\pgfusepath{stroke}
\end{pgfscope}
\begin{pgfscope}
\pgftransformcm{1.0}{-0.0}{0.0}{1.0}{\pgfpoint{125.572266bp}{521.950311bp}}
\pgftext[center,center]{\sffamily\mdseries\upshape\normalsize\color[rgb]{0.0,0.0,0.0}$\vdots$}
\end{pgfscope}
\begin{pgfscope}
\pgftransformcm{1.0}{-0.0}{0.0}{1.0}{\pgfpoint{126.664062bp}{586.950311bp}}
\pgftext[center,center]{\sffamily\mdseries\upshape\normalsize\color[rgb]{0.0,0.0,0.0}$\vdots$}
\end{pgfscope}
\begin{pgfscope}
\pgfsetlinewidth{0.5bp}
\pgfsetrectcap
\pgfsetmiterjoin
\pgfsetmiterlimit{10.0}
\pgfsetdash{{2.0bp}{4.0bp}}{0.0bp}
\pgfpathmoveto{\pgfpoint{100.0bp}{581.889764bp}}
\pgfpathlineto{\pgfpoint{100.0bp}{571.889764bp}}
\definecolor{strokepaint}{rgb}{0.0,0.0,0.0}\pgfsetstrokecolor{strokepaint}
\pgfusepath{stroke}
\end{pgfscope}
\begin{pgfscope}
\pgfsetlinewidth{0.5bp}
\pgfsetrectcap
\pgfsetmiterjoin
\pgfsetmiterlimit{10.0}
\pgfsetdash{{2.0bp}{4.0bp}}{0.0bp}
\pgfpathmoveto{\pgfpoint{150.0bp}{581.889764bp}}
\pgfpathlineto{\pgfpoint{150.0bp}{571.889764bp}}
\definecolor{strokepaint}{rgb}{0.0,0.0,0.0}\pgfsetstrokecolor{strokepaint}
\pgfusepath{stroke}
\end{pgfscope}
\begin{pgfscope}
\pgftransformcm{1.0}{-0.0}{0.0}{1.0}{\pgfpoint{226.664062bp}{521.950311bp}}
\pgftext[center,center]{\sffamily\mdseries\upshape\normalsize\color[rgb]{0.0,0.0,0.0}$\vdots$}
\end{pgfscope}
\begin{pgfscope}
\pgfsetlinewidth{0.5bp}
\pgfsetrectcap
\pgfsetmiterjoin
\pgfsetmiterlimit{10.0}
\pgfsetdash{{2.0bp}{4.0bp}}{0.0bp}
\pgfpathmoveto{\pgfpoint{200.0bp}{531.889764bp}}
\pgfpathlineto{\pgfpoint{200.0bp}{511.889764bp}}
\definecolor{strokepaint}{rgb}{0.0,0.0,0.0}\pgfsetstrokecolor{strokepaint}
\pgfusepath{stroke}
\end{pgfscope}
\begin{pgfscope}
\pgfsetlinewidth{0.5bp}
\pgfsetrectcap
\pgfsetmiterjoin
\pgfsetmiterlimit{10.0}
\pgfsetdash{{2.0bp}{4.0bp}}{0.0bp}
\pgfpathmoveto{\pgfpoint{250.0bp}{531.889764bp}}
\pgfpathlineto{\pgfpoint{250.0bp}{511.889764bp}}
\definecolor{strokepaint}{rgb}{0.0,0.0,0.0}\pgfsetstrokecolor{strokepaint}
\pgfusepath{stroke}
\end{pgfscope}
\begin{pgfscope}
\pgftransformcm{1.0}{-0.0}{0.0}{1.0}{\pgfpoint{126.664062bp}{461.950311bp}}
\pgftext[center,center]{\sffamily\mdseries\upshape\normalsize\color[rgb]{0.0,0.0,0.0}$\vdots$}
\end{pgfscope}
\begin{pgfscope}
\pgfsetlinewidth{0.5bp}
\pgfsetrectcap
\pgfsetmiterjoin
\pgfsetmiterlimit{10.0}
\pgfsetdash{{2.0bp}{4.0bp}}{0.0bp}
\pgfpathmoveto{\pgfpoint{150.0bp}{471.889764bp}}
\pgfpathlineto{\pgfpoint{150.0bp}{461.889764bp}}
\definecolor{strokepaint}{rgb}{0.0,0.0,0.0}\pgfsetstrokecolor{strokepaint}
\pgfusepath{stroke}
\end{pgfscope}
\begin{pgfscope}
\pgfsetlinewidth{0.5bp}
\pgfsetrectcap
\pgfsetmiterjoin
\pgfsetmiterlimit{10.0}
\pgfsetdash{{2.0bp}{4.0bp}}{0.0bp}
\pgfpathmoveto{\pgfpoint{100.0bp}{471.889764bp}}
\pgfpathlineto{\pgfpoint{100.0bp}{461.889764bp}}
\definecolor{strokepaint}{rgb}{0.0,0.0,0.0}\pgfsetstrokecolor{strokepaint}
\pgfusepath{stroke}
\end{pgfscope}
\begin{pgfscope}
\pgftransformcm{1.0}{-0.0}{0.0}{1.0}{\pgfpoint{226.664062bp}{481.950311bp}}
\pgftext[center,center]{\sffamily\mdseries\upshape\normalsize\color[rgb]{0.0,0.0,0.0}$\vdots$}
\end{pgfscope}
\begin{pgfscope}
\pgfsetlinewidth{0.5bp}
\pgfsetrectcap
\pgfsetmiterjoin
\pgfsetmiterlimit{10.0}
\pgfsetdash{{2.0bp}{4.0bp}}{0.0bp}
\pgfpathmoveto{\pgfpoint{250.0bp}{491.889764bp}}
\pgfpathlineto{\pgfpoint{250.0bp}{481.889764bp}}
\definecolor{strokepaint}{rgb}{0.0,0.0,0.0}\pgfsetstrokecolor{strokepaint}
\pgfusepath{stroke}
\end{pgfscope}
\begin{pgfscope}
\pgfsetlinewidth{0.5bp}
\pgfsetrectcap
\pgfsetmiterjoin
\pgfsetmiterlimit{10.0}
\pgfsetdash{{2.0bp}{4.0bp}}{0.0bp}
\pgfpathmoveto{\pgfpoint{200.0bp}{491.889764bp}}
\pgfpathlineto{\pgfpoint{200.0bp}{481.889764bp}}
\definecolor{strokepaint}{rgb}{0.0,0.0,0.0}\pgfsetstrokecolor{strokepaint}
\pgfusepath{stroke}
\end{pgfscope}
\begin{pgfscope}
\pgftransformcm{1.0}{-0.0}{0.0}{1.0}{\pgfpoint{182.25bp}{487.872186bp}}
\pgftext[center,center]{\sffamily\mdseries\upshape\normalsize\color[rgb]{0.0,0.0,0.0}$\Phi_i$}
\end{pgfscope}
\begin{pgfscope}
\pgftransformcm{1.0}{-0.0}{0.0}{1.0}{\pgfpoint{172.25bp}{547.872186bp}}
\pgftext[center,center]{\sffamily\mdseries\upshape\normalsize\color[rgb]{0.0,0.0,0.0}$\Phi_j$}
\end{pgfscope}
\end{pgfpicture}
\]
            
			\item At stage s, if the requirement is in Phase $n+1$, say this requirement requires attention if there are computations 
			\[ \Phi_{i, s}^{f^{\cL_{e, s}}}(x) = C_s(x) \quad \text{with use } u_{n+1}\]
			and 
			\[ \Psi_{j, s}^{C_s}[0, \ldots, m_{n+1}] = f^{\cL_{e, s}}[0, \ldots, m_{n+1}] \quad \text{with use } v_{n+1} \]
			where $m_{n+1} \ge \max\{m_n, u_{n+1}\}$ such that all $f$-blocks that intersect the set of elements ${[0, \ldots, \max\{m_n, u_{n+1}\}]}$, and all elements less than them in $\cL_{e,s}$, are completely contained in $[0, \ldots, m_{n+1}]$.
			
			When this requirement acts, restrain $[0, \ldots, v_{n+1}]$ in $C$ and define $C_{s+1}(x) = 1-C_s(x)$. Finally, move the requirement to Phase $n+2$.
		\end{enumerate}
		\noindent \textbf{Construction of $C$}: At stage $s$ of the construction, consider the first $s$ requirements, in order of decreasing priority. If any requirement requires attention then the one with the highest priority acts according to its strategy, injuring and resetting all lower priority strategies. If no requirement acts, then initialize the lowest priority requirement that has not yet been initialized.
		
		\medskip
		
		\noindent \textbf{Verification}: It is not obvious that the construction is a finite injury construction, as even if a requirement $R_{e,i,j}$ is not injured, it appears on the surface that it might go through infinitely many phases. In this case, our approximation for $C$ would also not come to a limit. We will argue by induction on the requirements that each requirement acts only finitely many times and is eventually satisfied. This argument will use the fact that there is no infinite coding sequence.
		
		Consider a requirement $R_{e, i, j}$ and assume that after some stage it is no longer injured. If ${\cL}_e$ is partial (or not a linear order), then $R_{e, i, j}$ is automatically satisfied and will no longer act after some stage. So we may assume that $\cL_e$ is a linear order. If there is some $n$ such that the strategy enters Phase $n$ but never enters Phase $n+1$, then we are also done since $R_{e,i,j}$ will have acted only finitely many times and will also be satisfied. (Otherwise, if $\Phi_i^{f^{\cL_e}} = C$ and $\Phi_j^C = f^{\cL_e}$, then we would eventually enter Phase $n+1$.) So it is enough to show that the strategy enters finitely many phases as this will also show that it requires attention finitely many times. To do this we show that if, after a strategy is no longer injured, it enters Phase $n$ for arbitrarily large $n$ then we can produce an infinite $f$-coding sequence; in fact we produce a weak $f$-coding sequence and then use \Cref{lem:coding-sequence} to obtain an $f$-coding sequence.
		
		Let $s_n$ be the stage at which the strategy enters Phase $n$, if it exists, and recall that $v_n$ is the restraint placed at Phase $n$. We begin by arguing that the restraints are maintained, and that $C$ (up to the restraint) cycles back and forth between two possible configurations depending on whether the phase is odd or even.
		
		\begin{claim}\label{lem:stages}
			Let $n' > n$ be of the same parity. Then:
			\begin{enumerate}
				\item $C_{s_{n'}} [0,\ldots,v_n] = C_{s_n} [0,\ldots,v_n]$.
				\item $f^{\cL_{e,s_{n'}}}[0,\ldots,m_n] = f^{\cL_{e,s_{n}}}[0,\ldots,m_n]$.
			\end{enumerate} 
			Further, for all $n$,
			\begin{enumerate}[resume]
				\item $f^{\cL_{e,s_{n}}}[0,\ldots,u_n] \neq f^{\cL_{e,s_{n+1}}}[0,\ldots,u_n]$.
			\end{enumerate}
		\end{claim}
		\begin{proof}
			We check each of (1), (2), and (3).
			\begin{enumerate}
				\item After Phase $n$, we have restrained the elements $[0, \ldots, v_n]$ and so, since the requirement is no longer injured by higher priority arguments, the only elements in $[0, \ldots, v_n]$ that can change value is the original $x$ that was restrained for $R_{e, i, j}$. However, since $n$ and $n'$ have the same parity, the construction ensures that $C(x)$ (whether $x$ is in $C$) is the same as well. Hence,  $C_{s_{n'}} [0,\ldots,v_n] = C_{s_n} [0,\ldots,v_n]$, as desired.
				\item By construction we have $\Phi_j^{C_{s_n}}[0, \ldots, m_n] = f^{\cL_{e,s_{n}}}[0, \ldots, m_n]$ and  $\Phi_j^{C_{s_{n'}}}[0, \ldots, m_n] = f^{\cL_{e,s_{n'}}}[0, \ldots, m_n]$. Further, the first computation listed has use $v_n$. Since by (1), this use $v_n$ is the same at Phases $n$ and $n'$, the computations $\Phi_j^{C_{s_n}}[0, \ldots, m_n]$ and $\Phi_j^{C_{s_{n'}}}[0, \ldots, m_n]$ must be the same as well and so $f^{\cL_{e,s_{n'}}}[0,\ldots,m_n] = f^{\cL_{e,s_{n}}}[0,\ldots,m_n]$, as desired.
				\item By construction $\Phi_i^{f^{\cL_{e,s_{n}}}}(x) = C_{s_n}(x)$ and $\Phi_i^{f^{\cL_{e,s_{n+1}}}}(x) = C_{s_{n+1}}(x)$ but $C_{s_n}(x) \neq C_{s_{n+1}}(x)$ and so the use of the computation $\Phi_i^{f^{\cL_{e,s_{n}}}}(x)$ must have changed by stage $s_{n+1}$, otherwise we would recover the same computation. Since the use is  $u_n$, we must have $f^{\cL_{e,s_{n}}}[0,\ldots,u_n] \neq f^{\cL_{e,s_{n+1}}}[0,\ldots,u_n]$, as desired. \qedhere
			\end{enumerate}        
		\end{proof}
		
		
		
		It will be helpful to consider, at each stage $s$, the guess $\pi_s \colon \cL_{e,s} \to {(\omega,<)}$ at the isomorphism $\cL_{e} \to {(\omega,<)}$. Note that (by definition), $f^{\cL_{e,s}}$ is the image of $f$ under $\pi_s$.
		
		Finally, to produce the weak $f$-coding sequence we proceed as follows: 
		\begin{itemize}
			\item Let $[a_0, b_0]$ be the smallest initial segment of ${(\omega,<)}$ which contains all of the $\pi_{s_0}$-images of the elements $0,\ldots,u_0$ in $\cL_{e,s_0}$ and all of the blocks which they intersect. Note that all of $[a_0,b_0]$ is contained within the $\pi_{s_0}$-images of $[0,\ldots,m_0]$.
			
			\item Given $[a_i, b_i]$, let $[a_{i+1}, b_{i+1}]$ be the minimal interval in ${(\omega,<)}$ that contains all of the $\pi_{s_{i+1}}$-images of the elements $0,\ldots,\max\{m_{i},u_{i+1}\}$ in $\cL_{e,s_{i+1}}$ and all of the blocks which they intersect. Note that all of $[a_i,b_i]$ is contained within the $\pi_{s_{i+1}}$-images of $[0,\ldots,m_{i+1}]$.
			
			\item Define the maps $\varphi_i: [a_i, b_i] \to [a_{i+1}, b_{i+1}]$ by $\varphi_i = \pi_{i+1} \circ \pi_{i}^{-1}$. (The way to think about these maps is that, when $[a_i,b_i]$ was defined at stage $s_{i}$, these elements of ${{(\omega,<)}}$ corresponded to certain elements of $\cL_e$. However by stage $s_{i+1}$, we have seen more elements enter $\cL_e$, and now those elements of $\cL_e$ are instead in correspondence with other elements of ${(\omega,<)}$, namely some elements of $[a_{i+1},b_{i+1}]$. The map $f_i$ keeps track of how this correspondence has changed. The element of $\cL_e$ that corresponded to $n \in [a_i,b_i]$ now corresponds to $f_i(n) \in [a_{i+1},b_{i+1}]$.)
			
		\end{itemize}
		This sequence satisfies most of the requirements by construction. We just need to ensure that the $\vp_i$ are not $f$-preserving  but the $\vp_{i+1} \circ \vp_i$ are. 
		But this is exactly the content of Lemma \ref{lem:stages}. Hence, assuming we enter Phase $n$ for arbitrarily large $n$ we have produced an infinite weak coding sequence. From an infinite weak coding sequence, using Lemma \ref{lem:coding-sequence} we can produce an infinite coding sequence. Thus, by our assumption about the absence of such a sequence it follows that after a strategy is no longer injured it only enters finitely many phases.
	\end{proof}

From this theorem we can easily recover several previous results, such as Theorem 3.10 of \cite{BNW22}.

\subsection{An example}

Now that we have produced necessary and sufficient conditions for the degree spectrum of a unary function being either all c.e. degrees or all $\Delta_2^0$ degrees on a cone, we know exactly which conditions a unary function must satisfy to have an intermediate degree spectrum and can produce an example.

\begin{theorem}\label{ex:int-fun}
	There is a computable block function $f$ whose degree spectrum on a cone strictly contains the c.e.\ degrees and is strictly contained within the $\Delta^0_2$ degrees. Moreover, the base of the cone is the computable degree, and so the degree spectrum of this function strictly contains the c.e.\ degrees and is strictly contained within the $\Delta^0_2$ degrees.
\end{theorem}
\begin{proof}
	Let $L_k$ be the block corresponding to the loop of length $k$, i.e., the block isomorphic to $[1, \ldots, k] \to [1, \ldots, k]$ via $x \mapsto x+1$ for $x<k$ and $k \mapsto 1$. Consider the block function $f$ where the odd blocks are given by the sequence $L_1, L_2, L_3, L_4, L_5, \ldots$ and the even blocks are given by the sequence $L_1, L_1, L_2, L_1, L_2, L_3, \ldots$. An initial segment looks like:  
	\[L_1 + L_1 + L_2 + L_1 + L_3 + L_2 + L_4 + L_1 + L_5 + L_2 + L_6 + L_3 + L_7 + \cdots  \]
	This function is constructed so that it satisfies the following: 
	\begin{itemize}
		\item all blocks that occur in the function occur infinitely often,
		\item no two block types embed in another different block type,
		\item no two different block types that occur in the function have the same size, and
		\item no two blocks types are adjacent (in the same order) more than once.
	\end{itemize}
	Since each block occurs infinitely often the degree spectrum of $f$ on a cone must contain a non-c.e. degree. To show that the degree spectrum of $f$ on a cone does not contain all $\Delta_2^0$ degrees we show that there is no infinite $f$-coding sequence. Indeed we will show that any $f$-coding sequence has length at most five. By Theorems \ref{them:cond-for-c.e.} and \ref{thm:all-d2} this proves the theorem. (For the moreover statement, one must observe that for this particular $f$ no cone is required in the proofs of these theorems.)
	
	Consider any (possibly finite) $f$-coding sequence $[a_1, b_1], [a_2, b_2], \ldots$ and corresponding maps $\varphi_i$. We make the following definitions. 
	\begin{enumerate}
		\item Say that $l_1 < \ldots < l_p$ form a \textit{link} in an interval $[a_i,b_i]$ if they form a block of length $p$ in $[a_i,b_i]$. Say that this interval \textit{witnesses} this link. 
		\item Say that a link $l_1 < \ldots < l_p$ in $[a_i,b_i]$ is \textit{vulnerable} in $[a_j,b_j]$, $j > i$, if the images $\varphi_{i \mapsto j}(l_1) < \cdots < \varphi_{i \mapsto j}(l_p)$ are no longer contained in a single block and now intersect two or more different blocks.
		\item Say that a link $l_1 < \ldots < l_p$ in $[a_i,b_i]$ is \textit{broken} in $[a_j,b_j]$, $j > i$, if some element is inserted between the images $\varphi_{i \mapsto j}(l_1)$ and $\varphi_{i \mapsto j}(l_p)$, i.e., the $l_r$ are no longer adjacent after applying $\varphi_{i \mapsto j}$.
	\end{enumerate}
	We first claim that either there is a link witnessed in $[a_1,b_1]$ that becomes vulnerable in $[a_2,b_2]$, or there is a link witnessed in $[a_2,b_2]$ that becomes vulnerable in $[a_3,b_3]$. The map $\varphi_1$ from $[a_1,b_1]$ to $[a_2,b_2]$ is not $f$-preserving, and so there is a link $l_1 < \cdots < l_p$ witnessed in $[a_1,b_1]$ such that $\varphi_1$ is not $f$-preserving on these elements. This link forms a block of size $p$ in $[a_1,b_1]$, and $f$ has only one block type of each size, so either the images of $l_1 < \cdots < l_p$ in $[a_2,b_2]$ lie in two different blocks or are contained in some larger block. In the former case, the link $l_1 < \cdots < l_p$ has become vulnerable in $[a_2,b_2]$. In the second case, let $k_1 < \cdots < k_q$ be this larger block in $[a_2,b_2]$ containing the images of $l_1,\ldots,l_p$. Then $k_1 < \cdots < k_q$ is a link witnessed in $[a_2,b_2]$ which becomes vulnerable in $[a_3,b_3]$. This is because $\varphi_{1 \mapsto 3}$ from $[a_1,b_1]$ to $[a_3,b_3]$ is $f$-preserving and so the images of $l_1 < \cdots < l_p$ form a block in $[a_3,b_3]$; this block is contained in, but not all of, the image of $k_1 < \cdots < k_q$ in $[a_3,b_3]$.
	
	Now suppose that we have a link $l_1,\ldots,l_p$ which is witnessed in $[a_i,b_i]$ and which becomes vulnerable in $[a_{i+1},b_{i+1}]$. We claim that it must break in either $[a_{i+2},b_{i+2}]$ or $[a_{i+3},b_{i+3}]$. Indeed, in $[a_{i+1},b_{i+1}]$ the images of the link $l_1,\ldots,l_p$ lie in at least two different blocks, and so the same is true in $[a_{i+3},b_{i+3}]$. Moreover, since the map $\varphi_{i+1 \mapsto i+3}\colon [a_{i+1},b_{i+1}] \to [a_{i+3},b_{i+3}]$ is $f$-preserving, in $[a_{i+3},b_{i+3}]$ the images of $l_1,\ldots,l_p$ lie in blocks of the same block types. However, since $\varphi_{i \mapsto i+2} \colon [a_{i},b_{i}] \to [a_{i+2},b_{i+2}]$ is $f$-preserving, the images of $l_1,\ldots,l_p$ form a block in $[a_{i+2},b_{i+2}]$. Thus---as each $\varphi_j$ is order-preserving---$\varphi_{i+1 \mapsto i+3}\colon [a_{i+1},b_{i+1}] \to [a_{i+3},b_{i+3}]$ cannot be the identity. Since each pair of $f$-blocks only appears adjacent to each other once in $f$, some element must be inserted between the images in $[a_{i+3},b_{i+3}]$ of $l_1$ and $l_p$; thus the link must be broken in $[a_{i+3},b_{i+3}]$ if not earlier.
	
	Now we argue that after a link is broken, the coding sequence must terminate. Suppose that a link $l_1,\ldots,l_p$ is witnessed in $[a_i,b_i]$ but broken in $[a_j,b_j]$. Then $i$ and $j$ must be of opposite parities, as if they were of the same parity then the map $\vp_{i \mapsto j} \colon [a_i,b_i] \to [a_j,b_j]$ would be $f$-preserving, and the images of $l_1,\ldots,l_p$ would form a block in $[a_j,b_j]$, and hence would be a sequence of successive elements. Thus, as $i$ and $j$ have opposite parities, the map $\varphi_{i \mapsto j+1}$ is $f$-preserving, so the images of $l_1,\ldots,l_p$ in $[a_{j+1},b_{j+1}]$ must form a block of size $p$ as $l_1,\ldots,l_p$ did in $[a_i,b_i]$. However, in $[a_j,b_j]$ and hence in $[a_{j+1},b_{j+1}]$ the images of $l_1,\ldots,l_p$ are no longer successive elements, preventing them from forming such a block. Thus the coding sequence must terminate with $[a_j,b_j]$.
	
	Thus we have a link in $[a_1,b_1]$ or $[a_2,b_2]$ which becomes vulnerable in $[a_2,b_2]$ or $[a_3,b_3]$, and hence broken by $[a_5,b_5]$ or before. There can be no more elements of the coding sequence.
\end{proof}

\section{Many different intermediate spectra}

Now that we know that intermediate degree spectrum are possible, we may ask for a more specific description of these intermediate degree spectra. How many possible intermediate degree spectra are there? Are particular interesting classes of degrees possible degree spectra? In this section, we continue to limit the relations we consider to block functions since this is the simplest case. Even so, we find significant complexity. We will particularly focus on the $\alpha$-c.e.\ degrees, whose definition we now recall.

\begin{definition}
	Let $(\alpha+1,\prec)$ be (a presentation of) a computable ordinal. A set $A$ is \textit{$\alpha$-c.e.}\ if there is a computable approximation function $g \colon \omega^2 \to \{0,1\}$ and a computable counting function $r \colon \omega^2 \to \alpha+1$ such that 
	\begin{enumerate}
		\item for all $x$, $g(x,0) = 0$,
		\item for all $x$, $\rank(x,0) = \alpha$,
		\item for all $x$ and $s$, $\rank(x,s+1) \preceq \rank(x,s)$,
		\item if $g(x,s+1) \neq g(x,s)$ then $\rank(x,s+1) \prec \rank(x,s)$, and
		\item $A(x) = \lim_{s \to \infty} g(x,s)$.
	\end{enumerate}
	We say that a set is of $\alpha$-c.e.\ degree if it is Turing equivalent to an $\alpha$-c.e.\ set.
\end{definition}
\noindent The definition of the $\alpha$-c.e.\ sets depends on the chosen the presentation of $\alpha+1$; indeed, for any $\Delta^0_2$ set, there is some presentation of $\omega^2+1$ which makes that set $\omega^2$-c.e. But for any fixed presentation of $\omega^2+1$, there are $\Delta^0_2$ sets which are not $\omega^2$-c.e. Any ordinal has a computable presentation on a cone, so when working on a cone we do not need to assume that ordinals are computable.

\begin{definition}
    Let $\mc{A}$ be a structure and $R$ a relation on $\mc{A}$. Let $\alpha$ be a (presentation of) an ordinal. We say that:
    \begin{enumerate}
        \item The degree spectrum of $R$ on a cone contains all $\alpha$-c.e.\ degrees if there is an $X$ (which we may assume computes that presentation of $\alpha$) such that for all $Y \geq_T X$, the degree spectrum of $R$ relative to $Y$ contains all degrees above $Y$ which are $\alpha$-c.e.\ relative to $Y$.
        \item The degree spectrum of $R$ on a cone is contained within the $\alpha$-c.e.\ degrees if there is an $X$ (which we may assume computes that presentation of $\alpha$) such that for all $Y \geq_T X$, the degree spectrum of $R$ relative to $Y$ is contained within the degrees above $Y$ which are $\alpha$-c.e.\ relative to $Y$.
    \end{enumerate}
\end{definition}

While these definitions involve a presentation of the ordinal, by the following remark the statements ``the degree spectrum of $R$ on a cone contains all $\alpha$-c.e.\ degrees'' and ``the degree spectrum of $R$ on a cone is contained within the $\alpha$-c.e.\ degrees'' are well-defined without fixing a presentation.

\begin{remark}\label{remark:wd}
    Let $\alpha$ be an ordinal, and let $(P,<_P)$ and $(Q,<_Q)$ be two presentations of $\alpha+1$. We write $\alpha_P$-c.e.\ for the $\alpha$-c.e.\ sets and degrees defined with respect to the presentation $P$, and similarly for $\alpha_Q$-c.e. There is a cone on which $(P,<_P)$ and $(Q,<_Q)$ are computable, and the $\alpha_P$-c.e.\ sets are the same as the $\alpha_Q$-c.e. sets. This cone is the join of the two presentations and an isomorphism between them. Thus working on a cone, it is reasonable to talk about the $\alpha$-c.e.\ sets without reference to presentation. This is a somewhat subtle point in that it is not true that there is a cone on which all presentations are equivalent, but rather that there is in some sense a single presentation up to equivalence in the limit.
    
    For example, let $R$ be a relation on a structure $\mc{A}$. Let $(P,<_P)$ and $(Q,<_Q)$ be any two presentations of $\alpha+1$. Suppose that, on a cone, the degree spectrum of $R$ is equal to the $\alpha_P$-c.e.\ degrees. Then, on a cone, the degree spectrum of $R$ will also be equal to the $\alpha_Q$-c.e.\ degrees because on a cone the $\alpha_P$-c.e.\ sets are the same as the $\alpha_Q$-c.e. sets. Thus whether or not the degree spectrum of $R$ on a cone is equal to the $\alpha$-c.e.\ degrees is independent of the presentation of $\alpha$ chosen, and so we can just say that the degree spectrum of $R$ on a cone is equal to the $\alpha$-c.e.\ degrees.
\end{remark}

Recall that, while there is no listing of all $\Delta^0_2$ sets, for each (presentation of a) computable ordinal $\alpha$ there is a listing of all of the $\alpha$-c.e.\ sets. The following theorem rules out, for block functions, any degree spectrum that can be listed in such a way other than the c.e.\ sets.

\begin{theorem}\label{thm:no-list}
	Let $f$ be a block function which is not intrinsically computable and such that, after some initial segment, all blocks embed into some later block. Let $(X_e^A)_{e \in \omega}$ be, uniformly in $A$, an $A$-computable list of (computable approximations to) $\Delta^0_2(A)$ sets each of which computes $A$. Then, for all $A$ on a cone, the degree spectrum of $f$ relative to $A$ contains some degree which is not Turing equivalent to any set in the listing $(X_e^A)_{e \in \omega}$. 
\end{theorem}

\begin{corollary}
	Let $f$ be a block function whose degree spectrum on a cone strictly contains the c.e.\ degrees. Then, for every $\alpha$, the degree spectrum of $f$ on a cone contains a non-$\alpha$-c.e.\ degree.
\end{corollary}

\begin{proof}[Proof of Theorem \ref{thm:no-list}]
	As usual we work on a cone on which we can compute $\alpha_f$, and assume without loss of generality that all blocks embed into some later block. Suppose $(X_e)_{e \in \omega}$ is a listing as above with $X_{e,s}$ being the limit approximation to $X_e$. We construct a computable copy $\cA$ of ${(\omega, <)}$ via finite stages, satisfying the requirements 
	\[ R_{e, i, j}: \text{either $\Phi_i^{f^{\cA}}\neq X_{e}$ or $\Phi_j^{X_{e}} \neq f^{\cA}$}\]
	The construction is a finite injury priority construction.
	
	The strategy for $R_{e, i, j}$ at stage $s$ is as follows
	\begin{enumerate}
		\item To initialize this requirement choose some interval $[a,b]$ which forms an $f$-block of size greater than one and such that $a$ is greater than the length of the finite linear order $\cA_{s-1}$. Insert new elements at the end of the linear order so that it has length $b$ and restrain the elements in the interval $[a,b]$ for this requirement. Call these elements $l_0, \ldots, l_{b-a}$. We say the requirement is in Phase 0. 
		\item At stage $s$, if the requirement is in Phase 0, say this requirement requires attention if there are computations 
		\[ \Phi_{i, s}^{X_{e, s}}[l_0, \ldots, l_{b-a}] = f^{\cA_s}[l_0, \ldots, l_{b-a}] \quad \text{with use } u_0\]
		and 
		\[ \Psi_{j, s}^{f^{\cA_s}}[0, \ldots, u_0] = X_{e, s}[0, \ldots, u_0] \quad \text{with use } v_0. \]
        
		When this requirement acts, restrain all blocks in $\cA_s$ which contain or are below some element of the use in the ordering $<_{\cA_s}$ other than the block $l_0,\ldots,l_{b-a}$. When a block is restrained, what this means is that $f$ on this block will never change. However, this does not mean that these elements of $\mc{A}$ have a fixed image in $(\omega,<)$: It may happen that the image in $(\omega,<)$ of a restrained block in $\mc{A}$ changes, with the new image being a block into which the old image embeds. The block in $\mc{A}$ becomes possibly larger. These are the only blocks that will be restrained by this requirement. We only need to maintain the values of $f$ on the elements that are in these blocks right now, and not the elements that are added to these blocks later.

        Now insert one element below the block consisting of the elements $l_0, \ldots, l_{b-a}$. Since they formed a block of size greater than one, they no longer form a block. Thus the value of $f$ on these elements has changed.
        
        For each block that was restrained, check to ensure the value of $f$ on these elements is the same as when it was restrained. If not, then insert new elements below the least element of this block and possibly between the elements of this block to move them up to the image of the original $f$-block in some other $f$-block into which it embeds. Further, ensure we add enough new elements on the end to complete this block. This, ensures that for all restrained elements, the value of $f^{\cA_{s+1}}$ on these elements is the same in stage $s$. We say this requirement is in Phase 1. 
		\item At stage $s$, if the requirement is in Phase $2n$ for $n \ge 1$,  say this requirement requires attention if there are computations 
		\[ \Phi_{i, s}^{X_{e, s}}[l_0, \ldots, l_{b-a}] = f^{\cA_s}[l_0, \ldots, l_{b-a}] \quad \text{with use } u_{2n}\]
		and 
		\[ \Psi_{j, s}^{f^{\cA_s}}[0, \ldots, u_{2n}] = X_{e, s}[0, \ldots, u_{2n}] \quad \text{with use } v_{2n} \]

        In Phase $2n$, the elements $l_0, \ldots, l_{b-a}$ will all be part of the same block (and indeed will be closed under $f$). When this requirement acts, insert enough new element below the block to make $l_0, \ldots, l_{b-a}$ now split across two blocks. This changes the value of $f$ on $l_0,\ldots,l_{b-a}$.
		
		For each block that was restrained in Phase 0 check to ensure the value of $f$ on these elements has not changed because of the addition of this new element. If it has, then insert new elements below the least element of this block and possibly between the elements of this block to move them up to the image of the original $f$-block in some other $f$-block into which it embeds. Further, ensure we add enough new elements on the end to complete this block. This ensures that for all restrained blocks, the value of $f^{\cA_{s+1}}$ on these elements at stage $s+1$ is the same as at stage $s$. We say this requirement is in Phase $2n+1$.
		\item At stage $s$, if the requirement is in Phase $2n+1$ for $n \ge 1$,  say this requirement requires attention if there are computations 
		\[ \Phi_{i, s}^{X_{e, s}}[l_0, \ldots, l_{b-a}] = f^{\cA_s}[l_0, \ldots, l_{b-a}] \quad \text{with use } u_{2n+1}\]
		and 
		\[ \Psi_{j, s}^{f^{\cA_s}}[0, \ldots, u_{2n+1}] = X_{e, s}[0, \ldots, u_{2n+1}] \quad \text{with use } v_{2n+1} \]
        During Phase $2n$, $l_0,\ldots,l_{b-a}$ were all part of the same block, while during Phase $2n+1$ they were split across different blocks. No elements have were inserted in between $l_0,\ldots,l_{b-a}$ in transferring from Phase $2n$ to Phase $2n+1$. When this requirement acts, insert enough new elements element below the element $l_0$ and possibly between the $l_0, \ldots, l_{b-a}$ so that they are moved up to the image of the block containing $l_0, \ldots, l_{b-a}$ during Phase $2n$ in some block into which it embeds. This ensures that the value of $f$ on $l_0, \ldots, l_{b-a}$ is the same as it was in Phase $2n$, and hence different from what it was in Phase $2n+1$.

        For each block that was restrained in Phase 0 check to ensure the value of $f$ on these elements has not changed because of the addition of these new elements. If it has, then insert new elements below the least element of this block and possibly between the elements of this block to move them up to the image of the original $f$-block in some other $f$-block into which it embeds. Further, ensure we add enough new elements on the end to complete this block. This ensures that for all restrained blocks, the value of $f^{\cA_{s+1}}$ on these elements at stage $s+1$ is the same as at stage $s$. We say this requirement is in Phase $2n+2$.
	\end{enumerate}
	
	\noindent
	Note that whenever the requirement is acted on we change the value of $f$ on the $l_i$, breaking the computation that was found at that stage and ensuring that the requirement is again satisfied when we move to the next stage. \\
	
	\noindent
	\textbf{Verification}: 
    We use the strategies described above in an injury argument. Since elements are inserted below the $l_i$ corresponding to some requirement only when that requirement or some higher priority requirement requires attention it is enough to check that after a requirement is no longer injured its strategy enters only finitely many phases. This will ensure that this construction produces a copy of ${(\omega, <)}$ and that all requirements are eventually satisfied, i.e., $\cA$ is not equivalent to any of the degrees in the listing.
	
	First, we make the following observations about the construction. Let $s_n$ be the stage where the strategy enters Phase $n$.
	\begin{itemize}
		\item The value of $f^{\cA_{s_{2n}}}$ is the same on the $l_i$ assigned to this strategy for all $n$. 
		\item The value of $f^{\cA_{s_n}}$ on the blocks which were restrained in Phase 0 is the same for all $n$.
	\end{itemize}
	We now show that $X_{e}$ changes each time we move to another phase but maintains the same value on some interval for all even phases. In the proof of Theorem \ref{thm:B} we showed the value of the set we were trying to beat oscillated between two possible values; in this case the set must only have a certain value for even phases but is allowed to vary on odd ones. 
	\begin{lemma} {\ }
		\begin{enumerate}
			\item For all even $n$, $X_{e, s_n}[0, \ldots, u_0] = X_{e, s_0}[0, \ldots, u_0]$
			\item For all $n$, $X_{e, s_n}[0, \ldots, u_0] \neq X_{e, s_{n+1}}[0, \ldots, u_0]$. 
		\end{enumerate}
	\end{lemma}
	\begin{proof}{\ }
		\begin{enumerate}
			\item Observe that for the computation $\Psi_{j,s_0}^{f^{\cA_{s_0}}}[0, \ldots, u_0] = X_{e, s_0}[0, \ldots, u_0]$ its use, except for possibly the $l_i$, is preserved at all stages and so since the value of $f^{\cA_{s_n}}$ is the same for all even $n$, the use of this computation must be preserved at stage $s_n$ for all even $n$ so the original computation must hold, i.e., $\Psi_{j,s_0}^{f^{\cA_{s_0}}}[0, \ldots, u_0] = \Psi_{j,s_0}^{\cA_{s_n}}[0, \ldots, u_n]$ and so it follows that $X_{e, s_n}[0, \ldots, u_0] = X_{e, s_0}[0, \ldots, u_0]$, as desired.

            \item The computation $\Phi_{i, s_0}^{X_{e, s}}[l_0, \ldots, l_{b-a}] = f^{\cA_{s_0}}[l_0, \ldots, l_{b-a}]$ found in Phase 0 has use $u_0$ and so by part (1) holds at stage $s_n$ for all even $n$. For each even $n$, $l_0, \ldots, l_{b-a}$ are part of the same block in $\mc{A}_{s_n}$ and closed under $f$, while for each odd $n$ they are not. Thus for all $n$ we have $f^{A_{s_n}}[l_0, \ldots, l_{b-a}] \neq f^{\cA_{s_{n+1}}}[l_0, \ldots, l_{b-a}]$. Since $\Phi_{i, {s_n}}^{X_{e, s}}[l_0, \ldots, l_{b-a}] = f^{\cA_{s_n}}[l_0, \ldots, l_{b-a}]$ for all $n$ we must have $\Phi_{i, s_n}^{X_{e, s_n}}[l_0, \ldots, l_{b-a}] \neq  \Phi_{i, s_{n+1}}^{X_{e, s_{n+1}}}[l_0, \ldots, l_{b-a}]$. Hence, the use of the computation must change between stage $s_n$ and $s_{n+1}$. Further, since either $n$ or $n+1$ is even, the use of one of these computations is $u_0$ and so we must have $X_{e, s_n}[0, \ldots, u_0] \neq X_{e, s_{n+1}}[0, \ldots, u_0]$, as desired.\qedhere
		\end{enumerate}
	\end{proof}
	
	\noindent
	This lemma tells us that some element in $[0, \ldots, u_0]$ must change value each time the strategy enters a new phase. Now, since $X$ is a $\Delta_2^0$ set, it follows that each finite collection of elements can change only finitely often so this strategy can enter only finitely many stages.  
\end{proof}

Despite the apparent difficulty in describing these intermediate degree spectra one way we can attempt to characterize them is by the least computable ordinal $\alpha$ such that their degree spectra contains all $\alpha$-c.e. degrees. (Recall that by Remark \ref{remark:wd} this is independent of the presentations chosen for the ordinals.) To do this we introduce the notion of coding trees to better understand which types of finite coding sequences exist for a given function $f$. To a given $f$ there will be associated two trees, a maximal coding tree and a minimal coding tree. The rank of the maximal coding tree will help to determine whether there is an $\alpha$-c.e.\ degree not contained in the degree spectrum, and the rank of the minimal coding tree will help to determine whether every $\alpha$-c.e.\ degree is contained in the degree spectrum.

The reason that we will need two different trees is that there are actually two ways in which we might want to be able to move coding elements in a linear order. The first, which was previously mentioned, is the ability to move the coding elements back and forth between two possible values of $f$ to account for changes in a $\Delta_2^0$ set. The second, which was not as apparent in the case where we had an infinite coding sequence, is the ability to preserve the value of $f$ on the coding sequence while other elements are inserted below them, i.e., the value of $f$ on other coding sequences is changing. So, the rank of the maximal coding tree captures which types of moves are possible when coding only a single element, which is enough for diagonalization, while the minimal coding tree captures which types of moves are possible while ensuring that coding sequences can move independently of each other, which is required for encoding a particular set. However, we will see below that the ranks of the minimal and maximal coding trees only provide an upper and lower bound on the $\alpha$ such that all $\alpha$-c.e. sets can be coded but it can be much more complicated to determine exactly which sets are in the degree spectrum.

We recall the rank of a tree.

\begin{definition}
	Let $T$ be a well-founded tree. Given $\sigma \in T$, the rank of $\sigma$ is defined by:
	\begin{enumerate}
		\item $\rank(\sigma) = 0$ if $\sigma$ is a leaf.
		\item $\rank(\sigma) = \sup\; \{\rank(\tau) + 1 \mid \text{$\tau$ child of $\sigma$}\}$.
	\end{enumerate}
	We define $\rank(T)$ to be the rank of the root node of $T$, or $\infty$ if $T$ is non-well-founded.
\end{definition}

\begin{definition}
	Let $f$ be a block function on ${(\omega,<)}$. We define its \textit{maximal coding tree} as follows. The elements of the tree will be all finite weak coding sequences, with the empty sequence as the root. Extension on the tree is just extension of coding sequences (including using the same functions $\vp$.) We call this tree $T_{\max}(f)$ and let the \textit{maximal coding rank} of $f$ be $\maxrank(f) = \rank(T_{\max}(f))$.
\end{definition} 


In Theorem \ref{thm:B} we showed that if a block function does not have an infinite $f$-coding sequence, then its degree spectrum on a cone is not all of the $\Delta^0_2$ degrees. Here we refine that using the maximal rank of a coding tree to say more about which $\Delta^0_2$ degree is not in the degree spectrum.

\begin{theorem} \label{thm:max-tree}
Let $f$ be a block function such that every block, except for finitely many, embeds into some later block. If $\alpha$ is the rank of the maximal coding tree for $f$, then on a cone there is some $\alpha$-c.e. degree which is not in the degree spectrum of $f$.
\end{theorem}

\begin{proof}
We work on a cone on which we can compute the usual facts about $f$ as well as a given presentation for $\alpha$ and the maximal coding tree for $f$ together with the ranking function on that tree using that presentation of $\alpha$ for the ranks. The construction here is similar to the construction in \Cref{thm:B}. There, we constructed a $\Delta^0_2$ set $C$ while diagonalizing against being Turing equivalent to $f^{\mc{A}_e}$ for copies $\mc{A}_e$ of ${(\omega,<)}$. We had requirements
\[ R_{e, i, j}: \text{If $\cL_e \cong {(\omega,<)}$, then either $\Phi_i^{f^{\cL_e}}\neq C$ or $\Psi_j^C \neq f^{\cL_e}$}.\]
For each requirement, we chose an element $x$ for which we would change $C(x)$. The fact that there was no infinite $f$-coding sequence meant that eventually each $R_{e, i, j}$ would be satisfied and $C(x)$ would no longer change; more precisely, we would only change $C(x)$ when we were able to extend the (weak) $f$-coding sequence that we were building.

The construction now will be exactly the same, except that we will add a rank function $r(x,s) \colon \omega^2 \to \alpha + 1$. Given $x$, we begin with $r(x,0) = \alpha$. We also maintain at each stage a weak coding sequence $\mathsf{CS}_{x,s}$ with rank in $T_{\max}(f)$ of $\rank(\mathsf{CS}_{x,s}) = r(x,s)$. We begin with $\mathsf{CS}_{x,0}$ as the empty coding sequence, with rank $\alpha = r(x,s)$ in $T_{\max}(f)$. If $x$ is never assigned to a requirement, then $C(x)$ never changes, and we never decrease $r(x,s)$. Otherwise, we do not change $r(x,s)$ and $\mathsf{CS}_{x,s}$ until after $x$ has been assigned to some requirement $R_{e, i, j}$, and then we will only have $r(x,s) < r(x,s-1)$ at stages $s$ when $R_{e, i, j}$ changes from one phase to the next; these are also the only stages when $C_s(x) \neq C_{s-1}(x)$. The weak coding sequence $\mathsf{CS}_{x,s}$ will be the sequence built by $R_{e,i,j}$ in Theorem \ref{thm:all-d2}. Suppose that we have defined $r(x,s-1)$ and $\mathsf{CS}_{x,s-1}$ and that at stage $s$ the requirement $R_{e, i, j}$ changes from Phase $n$ to Phase $n+1$. The weak coding sequence $\mathsf{CS}_{x,s-1}$ will have been the sequence of length $n$ constructed at Phase $n$. In moving to Phase $n+1$, we extend this weak coding sequence to a sequence $\mathsf{CS}_{x,s}$ of length $n+1$ as described in the proof of Theorem \ref{thm:all-d2}. Define $r(x,s)$ to be the rank of $\mathsf{CS}_{x,s}$ in the maximal coding tree $T_{\max}(f)$. If, at any point, the requirement $R_{e,i,j}$ is injured, then we never again change $C(x)$ and so never again have to decrease $r(x,s)$. Thus we have shown that the set $C$ built is $\alpha$-c.e., proving the theorem.
\end{proof}

Note that the condition given in this theorem is sufficient but not necessary. Indeed, suppose that $f_k$ is the function obtained from $f$ by ``removing'' the first $k$ blocks of $f$. Then $f$ and $f_k$ have the same degree spectra, but it is possible that $\maxrank(f_k) < \maxrank(f)$.

\begin{example}
Consider the example $f$ from Theorem \ref{ex:int-fun}. There, we showed that any $f$-coding sequence has length at most $5$. Thus the rank of the maximal coding tree for $f$ is at most $6$ (since in the rank of the coding tree we include the empty sequence). Thus by Theorem \ref{thm:max-tree} on a cone there is a $6$-c.e.\ degree which is not in the degree spectrum of that function $f$. Since in that example we can carry out the argument of Theorem \ref{thm:max-tree} computably, not working on a cone, there is a 6-c.e.\ degree which is not in the degree spectrum of $f$.
\end{example}



We now turn to a condition for the reverse, that is, a coding tree that will allow us to code all $\alpha$-c.e.\ sets. This is somewhat more complicated.

\begin{definition}
Let $f$ be a block function on ${(\omega,<)}$. Suppose that (a) all but finitely many blocks embed into infinitely many later blocks, or equivalently, the degree spectrum of $f$ strictly contains the c.e.\ degrees on a cone, and (b) there is no infinite $f$-coding sequence, or equivalently the degree spectrum of $f$ is strictly contained within the $\Delta^0_2$ degrees on a cone.

First, we need an additional definition. Given two coding sequences $\sigma$ and $\tau$ which agree on all but the last interval, which are, say, $[a_n, b_n]$ for $\sigma$ and $[a_n',b_n']$ for $\tau$, we say that $\sigma$ permits $\tau$ if $a_n < a_n'$  and there is a nondecreasing embedding $\psi$ from $[a_n,b_n] \to [a_n',b_n']$ which preserves $f$ and commutes with the maps $\varphi_{n-1} \colon [a_{n-1},b_{n-1}] \to [a_n,b_n]$ and $\varphi_{n-1}' \colon [a_{n-1},b_{n-1}] \to [a_n',b_n']$.
\[ \xymatrix@R=15mm@C=15mm{\cdots\ar[r] & [a_{n-2},b_{n-2}] \ar[r]^{\varphi_{n-2}} & [a_{n-1},b_{n-1}]\ar[r]^{\varphi_{n-1}} \ar[dr]_{\varphi_{n-1}'} & [a_n,b_n] \ar[d]^{\psi} \\ &&& [a_n',b_n']}.\]
The fact that each $f$-block embeds into infinitely many later blocks implies that for any coding sequence $\sigma_0$ there is an infinite sequence $\sigma = \sigma_0, \sigma_1, \sigma_2, \ldots$ such that each $\sigma_i$ permits $\sigma_{i+1}$.

We define the \textit{minimal coding tree} $T_{\min}(f)$ of $f$ as follows. The elements are all finite coding sequences, with the empty sequence as the root. For the maximal coding tree, we used weak coding sequences, but here we use strong coding sequences. Extension on the tree is just extension of coding sequences (using the same functions $\varphi$). Since there is no infinite $f$-coding sequence this tree is well-founded.

Define, inductively, a ranking $\minrank$ on $T_{\min}(f)$. For a leaf $\sigma$, we put $\minrank(\sigma) = 0$. Given a non-leaf $\sigma \neq \varnothing$, we set $\minrank(\sigma) \geq \alpha$ if there is an infinite sequence $\sigma = \sigma_0, \sigma_1, \sigma_2, \ldots$ such that each $\sigma_i$ permits $\sigma_{i+1}$ and for each $\beta < \alpha$ and each $i = 0,1,2,\ldots$ there is $\tau_i$ a child of $\sigma_i$ such that $\minrank(\tau_i) \geq \beta$. Then as usual we set $\minrank(\sigma)$ to be the greatest $\alpha$ such that $\minrank(\sigma) \geq \alpha$. We put $\minrank(\varnothing) = \sup_{\sigma \neq \varnothing} \minrank(\sigma)+1$.


For each non-root-node $\sigma \neq \varnothing$, there is an infinite sequence $\sigma = \sigma_0,\sigma_1,\sigma_2,\ldots$ such that each $\sigma_i$ permits $\sigma_{i+1}$ and for each $i$ $\minrank(\sigma_i) \geq \minrank(\sigma)$. If $\minrank(\sigma) > 0$ this can be seen directly from the definition of $\minrank$. If $\minrank(\sigma) = 0$ this is due to the fact noted above that because each $f$-block embeds into infinitely many later blocks, for any coding sequence $\sigma_0$ there is an infinite sequence $\sigma = \sigma_0, \sigma_1, \sigma_2, \ldots$ such that each $\sigma_i$ permits $\sigma_{i+1}$.)

We define $\minrank(f)$ to be $\minrank(\varnothing)$ the empty coding sequence which is the root node of $T_{\min}(f)$.
\end{definition}

\begin{theorem} \label{thm:min-tree}
Let $f$ be a block function such that every block, except for finitely many, embeds into some later block. If $\alpha = \minrank(f)$ is the rank of the minimal coding tree corresponding to $f$, then the degree spectrum of $f$ contains all $\beta$-c.e. degrees for any $\beta < \alpha$ on a cone.
\end{theorem} 

\begin{proof}[Proof of Theorem \ref{thm:min-tree}]
We work on the cone above $\alpha_f$ and the minimal coding tree and its ranking function into a presentation of $\alpha$. Suppose $X$ is an $\beta$-c.e.\ set, with $\beta < \alpha$, given by functions $g$, $r$ where $g$ is the computable approximation of $X$ and $r$ is ranking function counting the number of mind changes. We construct a computable copy $\cA$ of ${(\omega, <)}$ such that $f^{\mc{A}}$ is Turing equivalent to $X$. The construction will be almost exactly the same as \Cref{thm:A} but we will be much more careful in how we use our coding sequences. In that construction, we had at each stage $s$ certain elements $\bar{v}_e[s]$ of $\mc{A}_s$ which we used to code the value of $X(e)$. In particular, $\bar{v}_e[s]$ corresponded, via the partial isomorphism $\pi_s \colon \mc{A}_s \to {(\omega,<)}$, to an interval $[a_i,b_i]$ in a coding sequence. In \Cref{thm:all-d2} these intervals $[a_i,b_i]$ were in a fixed infinite $f$-coding sequence. Now we work within the minimal coding tree. At each stage $s$, we will have a coding sequence $\mathsf{CS}_{e,s} \in T_{\min}(f)$ with $\minrank(\mathsf{CS}_{e,s}) \geq r(e,s)$. The length of $\mathsf{CS}_{e,s}$ will be equal to 1 more than the number of times $X(e)$ has changed value before stage $s$,
\[ \text{length}(\mathsf{CS}_{e,s}) = 1 + \#\{t < s \; \mid \; X_t(e) \neq X_{t+1}(e) \}.\]
Suppose that $\mathsf{CS}_{e,s}$ consists of intervals $[a_1,b_1],\ldots,[a_n,b_n]$ with $n= 1 + \#\{t < s \; \mid \; X_t(e) \neq X_{t+1}(e) \}$. Then we will have $\pi_s(\bar{v}_e[s]) = [a_n,b_n]$. 
If at stage $s+1$ we see that $X_{s+1}(e) \neq X_s(e)$ then we know that $r(e,s+1) < r(e,s)$ and so we can choose an extension $\mathsf{CS}_{e,s+1}$ of $\mathsf{CS}_{e,s}$ with $\minrank(\mathsf{CS}_{e,s+1}) \geq r(e,s+1)$ which adds one more interval $[a_{n+1},b_{n+1}]$. We will add elements to $\mc{A}$ so that at stage $s+1$, $\bar{v}_e[s]$ corresponds to the images in $[a_{n+1},b_{n+1}]$, under $\varphi_{n}$, of $[a_n,b_n]$:
\[ \pi_{s+1}(\bar{v}_e[s]) = \varphi_n(\pi_s(\bar{v}_e[s])). \]
The other possibility is that at stage $s+1$ we have $X_{s+1}(e) = X_s(e)$ but some elements have already been added to $\mc{A}$ below $\bar{v}_e[s]$ for the sake of some $e' < e$. In \Cref{thm:all-d2} we had an infinite $f$-coding sequence and so we had enough room to just move along the sequence to another position $[a_j,b_j]$ with $j$ of the same parity as $i$. Now, we have limited room; instead, we use the permitting aspect of the tree. There is some other coding sequence $\mathsf{CS}_{e,s+1}$ which is of the same rank and length as $\mathsf{CS}_{e,s}$, which agrees with $\mathsf{CS}_{e,s}$ except for the last interval $[a_n',b_n']$, and which is permitted by $\mathsf{CS}_{e,s}$. There is a non-decreasing and $f$-preserving map $\psi\colon[a_n,b_n] \to [a_n',b_n']$. Moreover, we can choose such a coding sequence with $a_n'$ large enough that we can insert elements into $\mc{A}$ so that
\[ \pi_{s+1}(\bar{v}_e[s]) = \psi(\pi_s(\bar{v}_e[s])). \]
This allows us to adjust $\mc{A}$ for $e$ without decreasing the rank of the coding sequence. Thus it is not the case that the coding sequences $\mathsf{CS}_{e,s}$ for $e$ are just growing by extensions; rather, there are also lateral moves when the last interval in $\mathsf{CS}_{e,s}$ changes. Having given the informal idea of the construction, we will now give the formal construction.

\medskip

\noindent \textbf{Construction}: 
As in \Cref{thm:all-d2}: For each $e$, starting from stage $e$ where they are defined, we will have the initial coding elements $\bar{u}_e \in \mc{A}$, the full coding segment $\bar{v}_e \in \mc{A}$ containing $\bar{u}_e$, and a restraint $r_e = \max_{\mc{A}} \bar{v}_e \in \mc{A}$ which is the $\mc{A}$-largest element of $\bar{v}$. We write each tuple in $\mc{A}$-increasing order. The initial coding elements $\bar{u}_e$ will never change once defined, while the coding segment $\bar{v}_e$ may have more elements added, and the restraint $r_e$ may increase in the $\mc{A}$ order. We will use $\bar{v}_e[s]$ and $r_e[s]$ to refer to the values at stage $s$. We will always have that if $e' \leq e$ then $\bar{v}_e > r_{e'}$ so that the coding segment for $e$ is free of the restraint of $e'$.

At each stage $s$ for each $e$ we will also have $n_e[s]$ which will be one more than the number of changes\footnote{It will be helpful to assume that $X_s(e) = 0$ for all $s \leq e+1$.} in $X(e)$,
\[ n_e[s] = 1 + \# \{ t < s \; \mid \; X_t(e) \neq X_{t+1}(e) \}\]
and a coding sequence $\mathsf{CS}_{e,s}$ of length $n_e[s]$ and $\minrank(\mathsf{CS}_{e,s}) \geq r(e,s)$. We will generally write $[a_1,b_1],\ldots,[a_n,b_n]$ and $\varphi_1,\ldots,\varphi_{n-1}$ for the corresponding intervals and maps, taking $n = n_e[s]$ and suppressing $e$ and $s$ for simplicity of notation. Note that $n_e[s]$ will be odd if and only if $X_s(e) = 0$ and even if and only if $X_s(e) = 1$.

Recall that at each stage $s$, given a tuple $\bar{a} \in \mc{A}_s$, $\pi_s(\bar{a})$ is the corresponding tuple in ${(\omega,<)}$. We also have $f^{\mc{A}_s}$ which is the function $f$ on $\mc{A}_s$.

From each stage to the next, we will maintain the following properties:
\begin{enumerate}
	\item At each stage $s \geq e$, the coding elements $\bar{v}_e$ of $\mc{A}$ correspond to $\pi_s(\bar{v}_e)$ in ${(\omega,<)}$ which is the last interval $[a_n,b_n]$ in $\mathsf{CS}_{e,s}$.
	\item Suppose that at stage $s \geq e$ the coding elements $\bar{v}_e[s]$ corresponded (via $\pi_s \colon \mc{A}_s \to {(\omega,<)}$) to the interval $[a_n,b_n]$ in the coding sequence $\mathsf{CS}_{e,s}$. At stage $s+1$:
	\begin{enumerate}
		\item if $X_{s+1}(e) \neq X_s(e)$, then at stage $s+1$ the coding elements $\bar{v}[s+1]$ correspond (via $\pi_{s+1} \colon \mc{A}_{s+1} \to {(\omega,<)}$) to the interval $[a_{n+1},b_{n+1}]$ from the coding sequence $\mathsf{CS}_{e,s+1}$ which extends $\mathsf{CS}_{e,s}$ by this interval. The coding elements $\bar{v}_e[s]$ at stage $s$ correspond at stage $s+1$ to the images, in the commutative diagram, of the elements of ${(\omega,<)}$ that they corresponded to at stage $s$. That is,
		\[ \pi_{s+1}(\bar{v}_e[s]) = \varphi_{i \mapsto j} (\pi_s(\bar{v}_e[s])).\]
		\item if $X_{s+1}(e) = X_s(e)$ then at stage $s+1$ the coding elements $\bar{v}[s+1]$ correspond (via $\pi_{s+1} \colon \mc{A}_s \to {(\omega,<)}$) to an interval $[a_{n}',b_{n}']$ which is the last interval of the coding sequence $\mathsf{CS}_{e,s+1}$ of the same length as $\mathsf{CS}_{e,s}$, with $\minrank(\mathsf{CS}_{e,s+1}) \geq \minrank(\mathsf{CS}_{e,s})$, and such that $\mathsf{CS}_{e,s}$ permits $\mathsf{CS}_{e,s+1}$. Let $\psi$ be the map witnessing this. The coding elements $\bar{v}_e[s]$ at stage $s$ correspond at stage $s+1$ to the images of the elements of ${(\omega,<)}$ that they corresponded to at stage $s$. That is,
		\[ \pi_{s+1}(\bar{v}_e[s]) = \psi (\pi_s(\bar{v}_e[s])).\]
	\end{enumerate}
	\item If, at stage $s$ an element $a \in \mc{A}$ is not a coding element (i.e., not part of any $\bar{v}_e[s]$), then at stage $s+1$ it is still not a coding element, and $f^{\mc{A}_s}(a) = f^{\mc{A}_{s+1}}(a)$.
	\item From stage $s \geq e$ to stage $s+1$, elements can only be added to $\mc{A}$ below $r_e$ if there is $e' \leq e$ which entered or left $X$ at stage $s+1$, i.e., with $e' \in X_{s} \triangle X_{s+1}$. Thus, if $X_{s+1} \Eres_{e} = X_s \Eres_e$, then the partial isomorphism $\pi_s \colon \mc{A}_s \to {(\omega,<)}$ does not change at stage $s+1$ on elements below $r_e$.
\end{enumerate}

\smallskip

\noindent\textit{Stage $s$:} At this stage in the construction our partial linear order $\mc{A}_s$ can be partitioned into a finite number of coding segments and padding blocks, each of which form an interval:
\[ \bar{p}_0 <_{\mc{A}} \bar{v}_0[s] <_{\mc{A}} \bar{p}_1 <_{\mc{A}} \bar{v}_1[s] <_{\mc{A}} \cdots .\]
We add new elements to $\mc{A}_s$ to code $X_{s+1}$.

Starting with $\bar{v}_0[s]$, and continuing in order with $\bar{p}_1$, $\bar{v}_1[s]$, and so on, we ensure that each of these segments still satisfy our requirements. However we must now take into account the fact that we may have already inserted elements.
\begin{itemize}
	\item Given a padding segment $\bar{p}_i$, if no elements have been inserted below this segment, then we do not have to do anything. If there have been other elements, then we must act as follows. Consider each set of elements which corresponded at stage $s$ to a block. We must make sure that those same elements correspond in $\mc{A}_{s+1}$ to a block of the same type. For each block making up $\bar{p}_i$, insert new padding elements below the least element of this block and possibly between the elements of this block to move them up to the image of the original $f$-block in some other $f$-block into which it embeds. Further, ensure we add enough padding elements on the end to complete this block. This ensures that for all padding elements present at stage $s$, the value of $f^{\cA_s}$ on these elements is the same as at stage $s+1$. 
	\item Given $\bar{v}_e[s]$, the coding segment corresponding to $e$, first check the value of $X_{s+1}(e)$. We have two cases.
	
	\begin{itemize}
		\item If $X_{s+1}(e) \neq X_s(e)$, then $r(e,s+1) < r(e,s)$. Choose a coding sequence $\mathsf{CS}_{e,s+1}$ extending $\mathsf{CS}_{e,s}$ with $\minrank(\mathsf{CS}_{e,s+1}) \geq r(e,s+1)$, which we can do as $r(e,s+1) < r(e,s) \leq \minrank(\mathsf{CS}_{e,s})$. Let $[a_{n+1},b_{n+1}]$ be the last interval in $\minrank(\mathsf{CS}_{e,s+1})$. We can choose $\mathsf{CS}_{e,s+1}$ such that $a_{n+1}$ is sufficiently large that, by adding elements between and below the elements of $\bar{v}_e[s]$, we can have the segment $\bar{v}_e[s]$ corresponds to the image, under $\varphi_{k \mapsto \ell}$, of $[a_k,b_k] = \pi_s(\bar{v}_e[s])$. Any new elements that were inserted and end up in the new interval are added to the collection of coding elements $\bar{v}_e[s+1]$ corresponding to $e$, otherwise the newly added elements are padding elements. Further, we ensure that enough new coding elements are after the final element to ensure we complete the entire interval, putting these elements in $\bar{v}_e[s+1]$ as well. Thus we have
		\[ \pi_{s+1}(\bar{v}_e[s]) = \varphi_{i \mapsto j} (\pi_s(\bar{v}_e[s])) \]
		and
		\[ \pi_{s+1}(\bar{v}_e[s+1]) = [a_{n+1},b_{n+1}].\]
		
		\item If $X_{s+1}(e) = X_s(e)$, then we first ask whether some element has already been inserted below $\bar{v}_e[s]$. If not, then we do not insert any further elements. If so, say $k$ elements have been inserted below $\bar{v}_e[s]$, then we act as follows. We have at this stage a coding sequence $\mathsf{CS}_{e,s}$ with last interval $[a_n,b_n]$. Look for a coding sequence $\mathsf{CS}_{e,s+1}$ of the same length and rank which is permitted by $\mathsf{CS}_{e,s}$, and which has last interval $[a_n',b_n']$ with $a_n' \geq a_n + k$. Let $\psi$ be the corresponding non-decreasing $f$-preserving function $[a_n,b_n] \to [a_n',b_n']$. By inserting elements below and between the elements of $\bar{v}_e[s]$, we can have that the segment $\bar{v}_e[s]$ corresponds, at stage $s+1$, to the image in $[a_n',b_n']$, under $\psi$, of $[a_n,b_n] = \pi_s(\bar{v}_e[s])$. Any new elements that were inserted and end up in the new interval are added to the collection of coding elements $\bar{v}_e[s+1]$ corresponding to $e$, otherwise the newly added elements are padding elements. Further, we ensure that enough new coding elements are after the final element to ensure we complete the entire interval, putting these elements in $\bar{v}_e[s+1]$ as well. Thus we have
		\[ \pi_{s+1}(\bar{v}_e[s]) = \psi (\pi_s(\bar{v}_e[s])) \]
		and
		\[ \pi_{s+1}(\bar{v}_e[s+1]) = [a_{n}',b_{n}'].\]
	\end{itemize}
	
	\item Finally, we introduce the coding elements corresponding to $s$. After we have ensured all previously added elements still satisfy the requirements check the value of $X_{s+1}(s)$ (which we may assume is $0$, $s \notin X_{s+1}$). Choose a coding sequence $\mathsf{CS}$ of length $1$ such that $\minrank(\mathsf{CS}) \geq \beta \geq r(s,s+1)$. Since there is an infinite sequence of such coding sequences each permitting the next, we can choose such a coding sequence $\mathsf{CS}$ such that the last interval $[a,b]$ of that coding sequence has $a$ greater than the length of the linear order $\mc{A}$ at this stage. Insert enough new padding elements to the end of $\mc{A}$ to extend it to have length $a$, then add $b-a+1$ new coding elements corresponding to the interval $[a,b]$. These are the initial coding elements $\bar{u}_s[s+1]$ for $s$, and make up the coding segment $\bar{v}_s[s+1]$ at this stage. Set $\mathsf{CS}_{s,s+1}$ to be this coding sequence  $\mathsf{CS}_{s,s+1} = \mathsf{CS}$.
\end{itemize}
This ends the construction.

\medskip

\noindent \textbf{Verification}: Much of the verification is the same as in \Cref{thm:all-d2}. For example, to see that $\cA$ is really a computable copy of ${(\omega, <)}$ we observe as before that for any fixed $a$ there is some $e$ such that elements are only inserted below $a$ when for some stage $s$ and some $e' \leq e$, $X_{s+1}(e) \neq X_s(e)$. Since $X$ is $\Delta^0_2$, these eventually come to a limit, and so only finitely many elements are inserted below $a$.

It remains to show that $f^{\cA} \equiv_T X$. This is a little more complicated than in \Cref{thm:all-d2} due to the lateral movement from one coding sequence to a coding sequence it permits. The analogue of \Cref{lem:comp-of-maps} is as follows, and is proved in the same way.

\begin{lemma}
	Suppose that $s < t$.
	\begin{enumerate}
		\item If, for all $s'$ with $s \leq s' \leq t$ we have $X_{s'}(e) = X_{s}(e) = X_t(e)$, then $\mathsf{CS}_{e,s}$ is the same length as and permits $\mathsf{CS}_{e,t}$, say as witnessed by the $f$-preserving function $\psi \colon [a_n,b_n] \to [a_n',b_n']$ on the last intervals of these coding sequences. Moreover,
		\[ [a_n',b_n'] \supseteq \pi_{t}(\bar{v}_e[s]) =  \psi (\pi_s(\bar{v}_e[s])) = \psi [a_n,b_n]. \]
		
		\item If $s' \geq s$ is the least stage with $X_{s'+1}(e) \neq X_{s'}(e)$ then $\mathsf{CS}_{e,s}$ is the same length (say $m$) as and is either equal to or permits $\mathsf{CS}_{e,s'}$, as witnessed by the $f$-preserving function $\psi \colon [a_m,b_m] \to [a_m',b_m']$ on the last intervals of these coding sequences, and $\mathsf{CS}_{e,t}$ (say of length $n > m$) strictly extends $\mathsf{CS}_{e,s'}$. (If $\mathsf{CS}_{e,s}$ and $\mathsf{CS}_{e,s'}$ are the same then $\psi$ is the identity.) Moreover,
		\[ [a_n,b_n] \supseteq \pi_{t}(\bar{v}_e[s]) = \varphi_{m \mapsto n} \circ \psi (\pi_s(\bar{v}_e[s])) = \varphi_{m \mapsto n} \circ \psi [a_m,b_m]. \]
		In particular, if $m$ and $n$ are the same parity then $\varphi_{m \mapsto n} \circ \psi$ is $f$-preserving.
	\end{enumerate}
\end{lemma}

First we show that $f^{\cA} \ge_T X$. Given some element $e \in X$ run the above construction, which is computable, until stage $e+1$ when the initial coding elements $\bar{u}_e$ corresponding to $e$ are added. Now, compute the value of $f^{\mc{A}}$ on these elements. If $f^{\mc{A}}$ is the same on $\bar{u}_e$ as it was at stage $e+1$, then $X(e) = X_e(e+1)$. Otherwise, if $f^{\mc{A}}$ is different on $\bar{u}_e$ than it was at stage $e+1$, then $X(e) = 1- X_e(e+1)$.

Finally, we show that $X \ge_T f^{\cA}$. Given some element $a \in \mc{A}$ run the above construction until $a$ is added to the linear order. If $a$ is added as a padding element then the construction ensures that $f^{\cA}(a)$ does not change and so we take the value at this stage. If $a$ is a coding element corresponding to $e$ then $f^{\cA}(a)$ takes on one of two values depending on $X(e)$. Say that $a$ first appears in $\mc{A}$ at a stage $s$ in a coding block $\bar{v}_e[s]$, corresponding at that stage to $[a_i,b_i]$ and with $X_s(e) = k$. If $X(e) = X_s(e) = k$, then $f^{\cA}(a)$ is the same as it was at stage $s$. Otherwise, suppose that $X(e) \neq X_s(e) = k$. Then, at some stage $t > s$, we find that $X(e) = X_t(e) = 1-k$. At this stage, $a$ is an element of $\bar{v}_e[t]$ which corresponds to $[a_j,b_j]$ with $j$ of a different parity from $i$. Then $f^{\cA}(a)$ is the same as it was at stage $t$.
\end{proof}

\begin{remark}
Indeed, the proof can be used to show that all $\alpha$-computable degrees are contained in the degree spectrum of $f$, where a set $A$ is $\alpha$-computable if 
there is a computable approximation function $g \colon \omega^2 \to \{0,1\}$ and a computable counting function $r \colon \omega^2 \to \alpha$ such that 
\begin{enumerate}
	\item for all $x$, $r(x,0) < \alpha$,
	\item for all $x$ and $s$, $r(x,s+1) \preceq r(x,s)$,
	\item if $g(x,s+1) \neq g(x,s)$ then $r(x,s+1) \prec r(x,s)$, and
	\item $A(x) = \lim_{s \to \infty} g(x,s)$.
\end{enumerate}
The difference is that we do not require that $g(x,0) = 0$, so that we may begin with $x \notin A$ or $x \in A$. An $\alpha$-c.e.\ set is $\alpha$-computable but not necessarily vice versa. For degrees, the difference is mostly important when $\alpha$ is a limit ordinal.
\end{remark}

\begin{example}
Consider the function $f$ from Theorem \ref{ex:int-fun}. It was shown that, given any $f$-coding sequence $[a_1, b_1], [a_2, b_2], \ldots$, there must a link which is vulnerable in either $[a_2, b_2]$ or $[a_3, b_3]$. We show that the minimal coding tree of $f$ has $\minrank \leq 3$ by showing any coding sequence with a vulnerable link has rank $0$. Consider the case of a coding sequence $[a_1,b_1],[a_2,b_2],[a_3,b_3]$ with a link $\ell_1,\ldots,\ell_k$ witnessed in $[a_2,b_2]$ which is vulnerable in $[a_3,b_3]$. (The case of a link witnessed in $[a_1,b_1]$ which is vulnerable in $[a_2,b_2]$ is similar.) Consider a coding sequence $[a_1, b_1],[a_2,b_2],[a_3',b_3']$ which is permitted by $[a_1, b_1], [a_2,b_2],[a_3,b_3]$. The image in $[a_3,b_3]$ under $\varphi_2$ of the link $\ell_1,\ldots,\ell_k$ in $[a_2,b_2]$ intersects, in $[a_3,b_3]$, at least two blocks. Since any two blocks appear adjacent to each other in $f$ at most once, the image of these blocks in $[a_3',b_3']$ under $\psi$ are not adjacent and so have some element between them. But $\psi \circ \varphi_2$ is the map $\varphi_2' \colon [a_2,b_2] \to [a_3',b_3']$ from the coding sequence $[a_1, b_1],[a_2,b_2],[a_3',b_3']$, and so in $[a_3',b_3']$ the image of $\ell_1,\ldots,\ell_k$ under $\varphi_2'$ intersects two blocks which have had some element, not in the image of the link, between them. That is, $\ell_1,\ldots,\ell_k$ is actually broken in $[a_3',b_3']$. Thus the coding sequence $[a_1, b_1],[a_2,b_2],[a_3',b_3']$ has no extensions and is of $\minrank$ zero. Since $[a_1, b_1], [a_2,b_2],[a_3,b_3]$ only permits coding sequences of $\minrank$ zero, it also has $\minrank$ zero. Thus $\minrank(f) \leq 3$ as any coding sequence with three elements has $\minrank$ zero.

It is not hard to show that $\minrank(f) \geq 3$ so that $\minrank(f) = 3$. Let $\ell < m < n$. Consider the coding sequences $[a_1,b_1]$ where $[a_1,b_1]$ has block type $L_\ell$, an $\ell$-loop. For any two such coding sequences, the lesser one permits the greater one. These coding sequences can be extended to coding sequences $[a_1,b_1],[a_2,b_2]$ where $[a_2,b_2]$ has block type $L_m$, and $m$-loop, say under the map $\varphi_1 \colon [a_1,b_1] \mapsto [a_2,a_2 + b_1 - a_1]\subseteq [a_2,b_2]$. For any two such coding sequences with the same first interval $[a_1,b_1]$, the lesser one permits the greater one. Finally, any such sequence $[a_1,b_1],[a_2,b_2]$ extends to a sequence $[a_1,b_1],[a_2,b_2],[a_3,b_3]$ where the image of $[a_2,b_2]$ in $[a_3,b_3]$ is made up of two blocks, one of block type $L_\ell$, and the other of block type $L_n$. $[a_3,b_3]$ may have further blocks in between these. This shows that any $[a_1,b_1],[a_2,b_2]$ as above has $\minrank$ one, and any $[a_1,b_1]$ as above has $\minrank$ two. Thus $\minrank(f) = 3$.
\end{example}

Theorems \ref{thm:max-tree} and \ref{thm:min-tree} only tell us that on a cone the degree spectrum of this function $f$ contains all 2-c.e. degrees but not all 6-c.e. degrees, but does not give any information about the degrees between these. A more ad-hoc argument can be given to show that, in fact, the degree spectrum of $f$ does not contain all 3-c.e. degrees. This argument relies on the fact that any sequence permitting a vulnerable sequence has rank 0 and so this can be accounted for ahead of time in the priority argument by reserving an element which will force the coding sequence produced by the function to move to one which is permitted by it. However, one could imagine this will not happen more generally for any function of $\minrank$ three as this argument depends on the number of times one sequence permits a vulnerable sequence before decreasing its maximal rank. In fact, we believe that one could produce an example of a function with minimal rank 3 and maximal rank 6 but whose degree spectrum contains all 4-c.e. degrees. So, the difference between the minimal and maximal ranks arises not from a difference in the lengths of the coding sequences, but from the more complicated behavior underlying which sequences permit each other. We believe that by coming up examples with different permitting behaviours, one might be able to find, e.g., incomparable degree spectra on a cone.

\begin{conjecture}
	There are block functions $f$ and $g$ on ${(\omega,<)}$ whose degree spectra on a cone are incomparable.
\end{conjecture}

We will give some examples in this paper but we will focus on cases in which the minimal and maximal coding trees match up at various ordinals $\alpha$. We say that an ordinal $\alpha$ is even if it is of the form $\lambda + n$ where $\lambda$ is a limit ordinal and $n \in \omega$ is even; otherwise, we say that $\alpha$ is odd, in which case it is of the form $\lambda + n$ with $n$ odd.

\begin{theorem}\label{thm:example-alpha-even}
Fix $\alpha \geq 6$ even. There is a block function $f$ on ${(\omega,<)}$ whose degree spectrum on a cone contains all of the $\beta$-c.e.\ degrees for $\beta < \alpha$ and does not contain all of the $\alpha$-c.e.\ degrees.
\end{theorem}

\begin{proof}
Let $T \subseteq \omega^{<\omega}$ be a tree of rank $\alpha$, with rank function $\rank \colon T \to \alpha \cup \{\alpha\}$. Thus $\rank(\varnothing) = \alpha$ is even. We may choose $T$ such that for each $\sigma \in T$, $\rank(\sigma)$ has the same parity as the length of $\sigma$, i.e., if $\sigma$ is of even length then $\rank(\sigma)$ is an even ordinal, and if $\sigma$ is of odd length then $\rank(\sigma)$ is an odd ordinal. 

Fix some G\"odel numbering $\ell$ of the elements of $\omega^{<\omega}$.
 For each $\sigma = \langle a_0,\ldots,a_n \rangle \in \omega^{<\omega}$, with $\sigma \neq \varnothing$, associate with $\sigma$ the following block type $B_\sigma$: 
\[ x_0 + L_{2^{\ell({\sigma \upharpoonright 0})+2}} + L_{2^{\ell({\sigma \upharpoonright 1})+2}}+ \cdots + L_{2^{\ell({\sigma \upharpoonright n})+2}} + x_1 + x_2+x_3\]
where if $|\sigma|$ is even we let 
\[ x_0 \to x_3 \to x_2 \to x_1 \to x_1\]
and if $|\sigma|$ is odd we let: 
\[ x_0 \to x_3 \to x_2 \to x_1 \to x_2\]
We call $x_0,x_1,x_2,x_3$ the sandwich elements of the block.

Notice that each of these blocks have different lengths (in base 2 their lengths have 1's in positions uniquely determined by the initial segments of $\sigma$) and that each of their lengths are even. Now, we define $f$ as follows: The even blocks are the sequence $L_1, L_3, L_5, \ldots$ of loops of odd lengths. The odd blocks are given by $B_{\sigma_{1}}, L_{j(1)}, B_{\sigma_{2}}, L_{j(1)}, \ldots $
where $\sigma_1, \sigma_2, \ldots$ is a recursive enumeration of the elements of $T$ where each element appears infinitely often. Similarly, $j$ is a recursive enumeration of the odd integers such that each element appears infinitely often but with the additional condition that $j(i) \le 2i+1$ for all $i$. As in \Cref{ex:int-fun}, $f$ satisfies the following properties: 
\begin{itemize}
	\item all blocks that occur in the function occur infinitely often
	\item no two different block types that occur in the function have the same size 
	\item no two blocks types are adjacent (in the same order) more than once
\end{itemize}
There was one property in \Cref{ex:int-fun} that now fails: previously, no block type embedded into any other block type, but now there are many block types which embed into each other. The reader should note that the first bullet point is the reason that we include the blocks $L_{j(i)}$ along the odd sequence of blocks, and is what allows the degree spectrum of $f$ on a cone to contain non-c.e.\ degrees. They play no other role in the proof. The requirement that $j(i) \le 2i+1$ helps ensure the third bullet point.

Notice that $T$ embeds into the minimal coding tree of $f$. Given $\sigma \in T$ we can associate to it the coding sequence $[a_1, b_1], \ldots, [a_{|\sigma|}, b_{|\sigma|}]$ where $[a_i, b_i]$ is an interval with $a_i > b_{i-1}$ and which is a block of the form $B_{\sigma \upharpoonright i}$. (Note that there is no block for the empty sequence, but rather the root of $T$ maps to the empty coding sequence.) The maps $\vp_i: B_{\sigma \upharpoonright i} \to B_{\sigma \upharpoonright(i+1)}$ are the natural maps sending each $x_j$ in one block to the corresponding $x_j$ in the other block, and sending each loop to the loop of the same length. For each $\sigma$ this will be an $f$-coding sequence; it is important here how we have defined $f$ on $x_0,\ldots,x_3$ differently for odd and even $\sigma$. Also these coding sequences will be in the minimal coding tree since each interval in the sequence is a single block and so, since each block appears infinitely often, is always permitted by some later copy of the block. By Theorem \ref{thm:min-tree}, the degree spectra of $f$ contains all $\beta$-c.e. degrees for $\beta < \alpha$.

We consider two possible types of coding sequences. In $[a_1,b_1]$ there must be some element $x$ such that $\varphi_1(f(x)) \neq f(\varphi_1(x))$. In $[a_1,b_1]$ this element is part of a block, say of type $I$. Consider the image of this block in $[a_2,b_2]$ under $\varphi_1$. If the image is split between two different blocks, then we have a link in $[a_1,b_1]$ that becomes vulnerable in $[a_2,b_2]$. Otherwise, the image is contained entirely in one block, say of type $J$. Now this block embeds, by $\varphi_2$, in $[a_3,b_3]$. If the image is split between two different blocks, then we have a link in $[a_2,b_2]$ that becomes vulnerable in $[a_3,b_3]$. Otherwise, the entire image is contained in a single block, say of type $K$. We will concentrate on this case, returning later to the two cases of a link that becomes vulnerable.

We may assume, without loss of generality by shrinking intervals, that $[a_0,b_0]$ is a block of type $I$, and that each $[a_{n+1},b_{n+1}]$ is the closure, under blocks, of the image of $[a_n,b_n]$ under $\varphi_n$. In particular, $[a_1,b_1]$ is a block of type $J$, and $[a_2,b_2]$ is a block of type $K$. Since no two distinct block types have the same size, $J$ must be larger than $I$, and so $K$ must be larger than $I$. In particular, $I$ must be a block of type $B_{\sigma_1}$ and $K$ of type $B_{\sigma_3}$ for some $\sigma_1 \prec \sigma_3$ whose lengths have the same parity. Let $x_0,x_1,x_2,x_3$ be the four sandwich elements from $I$ with either
\[ x_0 \to x_3 \to x_2 \to x_1 \to x_1\]
or
\[ x_0 \to x_3 \to x_2 \to x_1 \to x_2.\]
The only elements of this kind are the four sandwich elements from a block $B_\tau$. Thus the $f$-preserving map $\varphi_{1 \mapsto 3} \colon [a_1,b_1] \to [a_3,b_3]$ must map the sandwich elements of $I$ to those of $K$, i.e.,
\[ a_1 \mapsto a_3, \qquad b_1-2 \mapsto b_3-2, \qquad b_1-1 \mapsto b_3 - 1, \qquad b_1 \mapsto b_3.\]
Then $\varphi_1$ must map the sandwich elements of $I$ to those of $J$,
\[ a_1 \mapsto a_2, \qquad b_1-2 \mapsto b_2-2, \qquad b_1-1 \mapsto b_2 - 1, \qquad b_1 \mapsto b_2,\]
and $\varphi_2$ must map the sandwich elements of $J$ to those of $K$,
\[ a_2 \mapsto a_3, \qquad b_2-2 \mapsto b_3-2, \qquad b_2-1 \mapsto b_3 - 1, \qquad b_2 \mapsto b_3.\]
(For example, if $a_1$ did not map to $a_2$, then we would have $\varphi_1(a_1) > a_2$ and so $a_3 = \varphi_2(\varphi(a_1)) > \varphi_2(a_2) \geq a_3$.)
That is, in the maps $I \to J \to K$ induced by $\varphi_1$ and $\varphi_2$, the sandwich elements are mapped to each other. The map $\varphi_{2 \mapsto 4} \colon [a_2,b_2] \to [a_4,b_4]$, since it preserves $f$, must map the sandwich elements of $J$ to the sandwich elements of another block, say of type $M$. Moreover, arguing similarly, $\varphi_3$ must map the elements of $K$ inside of $M$ while mapping sandwich elements to sandwich elements. Continuing in this way, the sequence $[a_1,b_1],[a_2,b_2],\ldots$ must consist of blocks of the form $B_{\sigma_1},B_{\tau_1},B_{\sigma_2},B_{\tau_2},\ldots$ where $\sigma_1 \prec \sigma_2 \prec \cdots$ and $\tau_1 \prec \tau_2 \prec \cdots$. Moreover, the lengths of the $\sigma_i$ are all of the same parity as each other, and the lengths of the $\tau_i$ are also of the same parity as each other. None of the $\sigma_i$ and $\tau_i$ are $\varnothing$, the root of $T$. Thus in $T$, $\rank(\sigma_i),\rank(\tau_i) < \alpha$. 

Now we return to the cases we skipped where there is a link which becomes vulnerable, and argue that any such $f$-coding sequence has length at most $5$:
For any $f$-coding sequence $[a_1, b_1], [a_2, b_2], \ldots$ and maps $\vp_i$, as shown in \Cref{ex:int-fun}, if there is some link in $[a_i, b_i]$ which is vulnerable in $[a_j, b_j]$ then the link will be broken by $[a_{j+2}, b_{j+2}]$ and so the coding sequence must terminate.

Thus, in the maximal coding tree, we have certain (finite) maximal paths of length $\leq 6$, e.g., those consisting of an $f$-coding sequence where a link becomes vulnerable,\footnote{All $f$-coding sequences have length 5, but in the tree this path starts with the empty $f$-coding sequence and so this corresponds to a path of length $6$.} and all other paths are induced by a pair of paths through $T$. This will imply, since $\alpha \geq 5$, that the rank of the maximal coding tree is at most $\alpha$. Indeed, consider the tree $T^*$ of pairs of finite sequences $\langle \sigma_1,\tau_1,\sigma_2,\tau_2,\ldots, \rangle$ of non-root elements of $T$ such that (a) $\sigma_1 \prec \sigma_2 \prec \cdots$ and $\tau_1 \prec \tau_2 \prec \cdots$, (b) the lengths of the $\sigma_i$ are all of the same parity as each other, and (c) the lengths of the $\tau_i$ are also of the same parity as each other. By (b) and (c), the ranks of the $\sigma_i$ are either all even or all odd, and similarly for the $\tau_i$. Since each $\sigma_i$ and $\tau_i$ is not the root $\varnothing$ of $T$, their ranks are all $< \alpha$. Let $\rank^*$ be the rank function on this tree $T^*$, which is clearly well-founded. Other than the $f$-coding sequences where a link becomes vulnerable, all $f$-coding sequences correspond to paths through $T^*$, and so by bounding the rank of $T^*$ we can bound the rank of the maximal coding tree. We prove the following four claims to calculate $\rank^*(T^*)$.

\begin{claim}
    If $\rank(\sigma_n)$ is even then 
	\[\rank^*(\langle \sigma_1,\tau_1,\sigma_2,\tau_2,\ldots,\tau_{n-1},\sigma_n \rangle) \leq \rank(\sigma_n)+1,\]
	and if $\rank(\sigma_n)$ is odd then
	\[\rank^*(\langle \sigma_1,\tau_1,\sigma_2,\tau_2,\ldots,\tau_{n-1},\sigma_n \rangle) \leq  \rank(\sigma_n).\]
	If $\rank(\tau_{n-1})$ is even then 
	\[\rank^*(\langle \sigma_1,\tau_1,\sigma_2,\tau_2,\ldots,\tau_{n-1},\sigma_n \rangle) \leq \rank(\tau_{n-1}),\]
	and if $\rank(\tau_{n-1})$ is odd then
	\[\rank^*(\langle \sigma_1,\tau_1,\sigma_2,\tau_2,\ldots,\tau_{n-1},\sigma_n \rangle) \leq \rank(\tau_{n-1})-1.\]
\end{claim}
\begin{proof}
    We argue inductively. Suppose that we have $\langle \sigma_1,\tau_1,\sigma_2,\tau_2,\ldots,\tau_{n-1},\sigma_n,\tau_n \rangle$ an extension of rank $\beta$   extending $\langle \sigma_1,\tau_1,\sigma_2,\tau_2,\ldots,\tau_{n-1},\sigma_n \rangle$. We have several cases to check.
\begin{itemize}
	\item Suppose $\rank(\tau_{n-1})$ and $\rank(\tau_{n})$ are both even. Then, inductively, $\beta \leq \rank(\tau_n)+1$. Since $\tau_{n}$ is an extension of $\tau_{n-1}$ and $\rank(\tau_n)$ is at least two less than $\rank(\tau_{n-1})$, $\beta + 1 \leq \rank(\tau_n) + 2 \leq \rank(\tau_{n-1})$ and so
	\[ \rank^*(\langle \sigma_1,\tau_1,\sigma_2,\tau_2,\ldots,\tau_{n-1},\sigma_n \rangle) \leq \rank(\tau_{n-1}).\]
	\item Suppose that $\rank(\tau_{n-1})$ and $\rank(\tau_n)$ are both odd. Then, inductively, $\beta \leq \rank(\tau_n)$. Since $\rank(\tau_{n})$ is at least two less than $\rank(\tau_{n-1})$, $\beta + 1 \leq \rank(\tau_n) + 1 \leq \rank(\tau_{n-1}) - 1$ and so
	\[ \rank^*(\langle \sigma_1,\tau_1,\sigma_2,\tau_2,\ldots,\tau_{n-1},\sigma_n \rangle) \leq \rank(\tau_{n-1})-1.\]
	\item Suppose that $\rank(\sigma_n)$ is even. Then, inductively, $\beta \leq \rank(\sigma_n)$ and so $\beta+1 \leq \rank(\sigma_n)+1$. Thus
	\[ \rank^*(\langle \sigma_1,\tau_1,\sigma_2,\tau_2,\ldots,\tau_{n-1},\sigma_n \rangle) \leq \rank(\sigma_n)+1.\]
	\item Suppose that $\rank(\sigma_n)$ is odd. Then, inductively, $\beta \leq \rank(\sigma_n)-1$ and so $\beta+1 \leq \rank(\sigma_n)$. Thus
	\[ \rank^*(\langle \sigma_1,\tau_1,\sigma_2,\tau_2,\ldots,\tau_{n-1},\sigma_n \rangle) \leq \rank(\sigma_n).\qedhere\]
\end{itemize}
\end{proof}

\begin{claim}
    If $\rank(\tau_n)$ is even then 
	\[\rank^*(\langle \sigma_1,\tau_1,\sigma_2,\tau_2,\ldots,\sigma_n,\tau_n \rangle) \leq \rank(\tau_n)+1,\]
	and if $\rank(\tau_n)$ is odd then
	\[\rank^*(\langle \sigma_1,\tau_1,\sigma_2,\tau_2,\ldots,\sigma_n,\tau_n \rangle) \leq  \rank(\tau_n).\]
	If $\rank(\sigma_n)$ is even then 
	\[\rank^*(\langle \sigma_1,\tau_1,\sigma_2,\tau_2,\ldots,\sigma_n,\tau_n \rangle) \leq \rank(\sigma_n),\]
	and if $\rank(\sigma_n)$ is odd then
	\[\rank^*(\langle \sigma_1,\tau_1,\sigma_2,\tau_2,\ldots,\sigma_n,\tau_n \rangle) \leq \rank(\sigma_n)-1.\]
\end{claim}
\begin{proof}
    This can be argued similarly to the previous claim.
\end{proof}

\begin{claim}
If $\rank(\sigma_1)$ is even then
	\[\rank^*(\langle \sigma_1 \rangle) \leq \rank(\sigma_1)+1,\] 
	and if $\rank(\sigma_1)$ is odd then
	\[\rank^*(\langle \sigma_1 \rangle) \leq \rank(\sigma_1)\]
\end{claim}
\begin{proof}
    We have two cases. Suppose $\langle \sigma_1,\tau_1 \rangle$ is an extension of $\langle \sigma_1 \rangle$ of rank $\beta$. 
\begin{itemize}
	\item If $\rank(\sigma_1)$ is even then $\beta \leq \rank(\sigma_1)$ and so $\beta + 1 < \rank(\sigma_1)+1$. Thus
	\[ \rank^*(\langle \sigma_1 \rangle) \leq \rank(\sigma_1)+1.\]
	\item If $\rank(\sigma_1)$ is odd then $\beta \leq \rank(\sigma_1) - 1$ and so $\beta + 1 \leq \rank(\sigma_1)$. Thus
	\[ \rank^*(\langle \sigma_1 \rangle) \leq \rank(\sigma_1).\qedhere\]
\end{itemize}
\end{proof}

\begin{claim}
    Recalling that $\alpha$ is even,
	\[\rank^*(\varnothing) \leq \alpha.\]
\end{claim}
\begin{proof}
    Suppose that $\langle \sigma_1 \rangle$ is an extension of $\varnothing$ of rank $\beta$. Then we have two cases.
\begin{itemize}
	\item If $\rank(\sigma_1)$ is odd, then $\beta \leq \rank(\sigma_1) < \alpha$. Thus $\beta+1 \leq \alpha$ and so
	\[ \rank^*(\varnothing) \leq \alpha.\]
	\item If $\rank(\sigma_1)$ is even then $\beta \leq \rank(\sigma_1) +1$. Since $\rank(\sigma_1)$ is even ordinal, and $\rank(\sigma_1) < \alpha$ where $\alpha$ is also an even ordinal. Thus $\rank(\sigma_1) +1 < \alpha$. Hence $\beta+1 \leq \alpha$ and so
	\[ \rank^*(\varnothing) \leq \alpha.\qedhere\]
\end{itemize}
\end{proof}
In particular, we have shown that $\rank^*(T^*) = \rank^*(\varnothing) \leq \alpha$. Thus, since $\alpha \geq 5$, we have shown that the rank of the maximal coding tree $T_{\max}(f)$ is at most $\alpha$. By Theorem \ref{thm:max-tree}, on a cone the degree spectrum of $f$ does not contain all $\alpha$-c.e.\ degrees.
\end{proof}

We do not know if this can be proved for $\alpha$ an odd ordinal.

\begin{question}
Is the statement of Theorem \ref{thm:example-alpha-even} provable for $\alpha$ an odd ordinal?
\end{question}

We end this paper with several other questions and conjectures. Though we state these for block functions, we are also interested more generally in any relation $R$ on ${(\omega,<)}$.

\begin{question}
	If the degree spectrum on a cone of a block function $f$ contains the non-c.e.\ degrees, must it contain all $2$-c.e.\ degrees?
\end{question}

\begin{question}
	Is there an infinite descending sequence of degree spectra on a cone of block functions?
\end{question}

\begin{question}
	Is there an exact classification of when the degree spectrum of a block function $f$ on a cone contains all of the $\alpha$-c.e.\ degrees?
\end{question}

\bibliography{References}
\bibliographystyle{alpha} 

\end{document}